\documentclass[a4paper,reqno,11pt]{amsart}
\usepackage{amsmath,amssymb,amsfonts,amsxtra}

\usepackage{xcolor}

\usepackage{tikz}
\usetikzlibrary{matrix,arrows}
\usepackage{verbatim}

\usepackage[pagebackref,colorlinks,citecolor=red,linkcolor=blue]{hyperref}

\renewcommand*{\backref}[1]{}
\renewcommand*{\backrefalt}[4]{%
    \scriptsize%
    {
    \ifcase #1 (\textcolor{red}{Uncited.})%
          \or (Cit.\ on p.~#2)%
          \else (Cit.\ on pp.~#2)%
    \fi%
    }
}

\makeatletter
\providecommand*{\twoheadrightarrowfill@}{%
  \arrowfill@\relbar\relbar\twoheadrightarrow
}
\providecommand*{\twoheadleftarrowfill@}{%
  \arrowfill@\twoheadleftarrow\relbar\relbar
}
\providecommand*{\xtwoheadrightarrow}[2][]{%
  \ext@arrow 0579\twoheadrightarrowfill@{#1}{#2}%
}
\providecommand*{\xtwoheadleftarrow}[2][]{%
  \ext@arrow 5097\twoheadleftarrowfill@{#1}{#2}%
}
\makeatother

\title[Formal De Rham complex]{On the decomposition of the De Rham complex on formal schemes}

\author[L. Alonso]{Leovigildo Alonso Tarr\'{\i}o}
\address[L. A. T.]{Instituto de Matem\'aticas\\ 
Departamento de  Matem\'a\-ticas\\
Universidade de Santiago de Compostela\\
E-15782  Santiago de Compostela, Spain}
\email{leo.alonso@usc.es}

\author[A. Jerem\'{\i}as]{Ana Jerem\'{\i}as L\'opez}
\address[A. J. L.]{Instituto de Matem\'aticas\\ 
Departamento de  Matem\'a\-ticas\\
Universidade de Santiago de Compostela\\
E-15782  Santiago de Compostela, Spain}
\email{ana.jeremias@usc.es}

\author[M. P\'erez]{Marta P\'erez Rodr\'{\i}guez}
\address[M. P. R.]{Departamento de Matem\'a\-ticas\\
Facultade de Ciencias da Educaci\'on e do Deporte,
Campus A Xunqueira\\
Universidade de Vigo\\
E-36005 Pontevedra, Spain}
\email{martapr@uvigo.es}

\thanks{This paper has been partially supported by  
Spain's MEIC and E.U.'s FEDER research projects MTM2014-59456
and MTM2017-89830-P together with Xunta de Galicia's ED431C 2019/10 with FEDER funds.}


\date{October 31, 2019, \emph{typeset}: \today}

\theoremstyle{plain}
\newtheorem{thm}{Theorem}[section]
\newtheorem{lem}[thm]{Lemma}
\newtheorem{cor}[thm]{Corollary}
\newtheorem{prop}[thm]{Proposition}

\theoremstyle{remark}
\newtheorem*{rem}{Remark}

\theoremstyle{definition}

\newtheorem*{ack}{Acknowledgements}
\newtheorem{cosa}[thm]{}

\numberwithin{equation}{thm}

\newcommand{\CA}{\mathcal{A}}

\newcommand{\CE}{{\mathcal E}}
\newcommand{\CF}{{\mathcal F}}
\newcommand{\CG}{{\mathcal G}}
\newcommand{\CH}{{\mathcal H}}
\newcommand{\CI}{{\mathcal I}}
\newcommand{\CJ}{{\mathcal J}}
\newcommand{\CK}{{\mathcal K}}

\newcommand{\CO}{\mathcal{O}}

\newcommand{\CZ}{\mathcal{Z}}

\newcommand{\FU}{\mathfrak U}
\newcommand{\FV}{\mathfrak V}
\newcommand{\FW}{\mathfrak W}
\newcommand{\FX}{\mathfrak X}
\newcommand{\FY}{\mathfrak Y}
\newcommand{\FZ}{\mathfrak Z}

\newcommand{\sch}{\mathsf {Sch}}

\newcommand{\sfn}{\mathsf {NFS}}

\newcommand{\sfna}{\sfn_{\mathsf {af}}}

\newcommand{\D}{\boldsymbol{\mathsf{D}}}

\newcommand{\R}{\boldsymbol{\mathsf{R}}}

\newcommand{\A}{\mathsf{A}}

\newcommand{\cc}{\mathsf{c}}

\newcommand{\ts}{\mathsf{t}}

\newcommand{\qc}{\mathsf{qc}}
\newcommand{\qct}{\mathsf{qct}}

\newcommand{\tr}{\triangle}
\newcommand{\tc}{\widehat{\otimes}}
\newcommand{\om}{\widehat{\Omega}}
\newcommand{\omi}{\widehat{\Omega}^{i}}
\newcommand{\pei}{\bigwedge^{i}}

\newcommand{\chC}{\Check{\mathcal C}}

\newcommand{\LLambda}{\boldsymbol{\Lambda}}

\newcommand{\hd}{\widehat{d}}

\newcommand{\HPR}{\widehat{\mathbb{P}}}

\newcommand{\PR}{\mathbb{P}}
\newcommand{\AF}{\mathbb{A}}

\newcommand{\NN}{\mathbb{N}}
\newcommand{\ZZ}{\mathbb{Z}}

\newcommand{\kk}{\kappa}

\newcommand{\dirlim}[1]{\begin{array}[t]{c} {\rm lim}\\[-7.5 pt]
 {\longrightarrow} \\[-7.5 pt] {\scriptstyle {#1}} \end{array}}
\newcommand{\invlim}[1]{\begin{array}[t]{c} {\rm lim}\\[-7.5 pt]
 {\longleftarrow} \\[-7.5 pt] {\scriptstyle {#1}} \end{array}}

\newcommand{\lto}{\longrightarrow}
\newcommand{\xto}{\xrightarrow}

\newcommand{\xepi}{\xtwoheadrightarrow}

\newcommand{\inc}{\hookrightarrow}

\newcommand{\liso}{\mathop{\tilde{\lto}}}

\newcommand{\ush}{\textup{\texttt\#}}

\DeclareMathOperator{\Hom}{Hom}
\DeclareMathOperator{\shom}{\CH\mathit{om}}
\DeclareMathOperator{\rshom}{\R\!\shom}

\DeclareMathOperator{\Img}{Im}

\DeclareMathOperator{\Dercont}{Dercont}

\DeclareMathOperator{\ga}{\Gamma}
\DeclareMathOperator{\spec}{Spec}
\DeclareMathOperator{\spf}{Spf}

\DeclareMathOperator{\h}{H}
\DeclareMathOperator{\id}{id}

\DeclareMathOperator{\rg}{rk}

\DeclareMathOperator{\dimtop}{dimtop}
\DeclareMathOperator{\modu}{\text{-}\mathsf{mod}}
\DeclareMathOperator{\res}{res}

\newcommand{\ie}{{\it i.e.\/} }
\newcommand{\cfr}{{\it cf.\/} }

\begin{document}

\begin{abstract} 
We show that if $\FX$ is a pseudo-proper smooth noetherian formal scheme over a positive characteristic $p$ field $k$ then its De Rham complex $\tau^{\leq p}(F_{\FX/k\,*} \om^{\bullet}_{\FX/k})$ is decomposable. Along the way we establish the Cartier isomorphism
$\om^{i}_{\FX^{(p)}/\FY} \overset{\gamma}\to \CH^{i} (F_{\FX/\FY\,*}\om^{\bullet}_{\FX/\FY})$ associated to a map $f\colon \FX \to \FY$ of positive characteristic $p$ noetherian formal schemes where $\FX^{(p)}$ denotes the base change of $\FX$ along the Frobenius morphism of $\FY$ and $F_{\FX/\FY}$ denotes the relative Frobenius of $\FX$ over $\FY$.
\end{abstract}

\maketitle
\tableofcontents

\section*{Introduction}
An important tool for understanding some of the fine properties of algebraic varieties is the use of formal schemes. Over the field of complex numbers, Hartshorne studied the hypercohomology of the De Rham complex of the formal completion of a singular scheme on a non-singular ambient scheme and showed that this gives back singular cohomology by purely algebraic means. 

In this paper we start exploring the properties of De Rham cohomology of formal schemes over a characteristic $p$ field. A motivation is to develop tools to understand the cohomological properties of singular varieties. The main technical issue is to have at hand basic results about the geometry of formal schemes. Let $X$ be a possibly singular variety over a field $k$. Suppose there is a closed embedding $X \inc P$ of $X$ into a smooth $k$-scheme $P$. Its formal completion $P_{/X}$ is not \emph{adic} over $\spec(k)$. This leads us to consider non-adic morphisms of formal schemes. Let $f \colon \FX \to \FY$ be a morphism of formal schemes. As explained in \ref{lim} (ii) there is a system of morphisms of usual schemes $\{f_{\ell}\colon  X_{\ell} \to Y_{\ell}\}_{\ell \in \NN}$ such that
\[f = \dirlim {\ell\in \NN} f_{\ell}.\]
It is a general principle that if $f$ is adic, its properties can be studied through the underlying maps $f_{\ell}$, after all, the squares
\begin{equation*}
\begin{tikzpicture}[baseline=(current  bounding  box.center)]
\matrix(m)[matrix of math nodes, row sep=2.4em, column sep=2.4em,
text height=1.5ex, text depth=0.25ex]{
X_{\ell} & Y_{\ell} \\
\FX   & \FY \\}; 
\path[->,font=\scriptsize,>=angle 90]
(m-1-1) edge node[auto] {$f_{\ell}$} (m-1-2)
        edge node[left] {$i'_\ell$} (m-2-1)
(m-1-2) edge node[auto] {$i_\ell$} (m-2-2)
(m-2-1) edge node[auto] {$f$} (m-2-2);
\end{tikzpicture}
\end{equation*}
can be taken Cartesian. This is \emph{not} the case for non-adic morphisms. Thus, one needs to redevelop most of the usual tools for non adic maps of formal schemes. To give a specific example, if $f$ is a smooth morphism  of locally noetherian formal schemes the morphisms $f_{\ell}$ may not be smooth (see \cite[Example 5.3]{AJP2}), therefore one cannot use a limit argument to reduce the arguments to ordinary schemes.

Here, we study the De Rham complex of a non necessarily adic formal scheme of pseudo finite type over a field of positive characteristic $p$. We show that under the usual condition of $W_2$-liftability the De Rham complex is decomposed up to $p$. The argument does not give the degeneration of the Hodge-De Rham spectral sequence because the finiteness of cohomology is only established under adic hypothesis. 

The strategy of the proof is similar to the classical method by Deligne and Illusie \cite{deill} but all the results of smoothness, deformation and cohomology are needed in the setting of pseudo-finite maps of formal schemes. The basic theory of smoothness of formal schemes is developed in \cite{AJP1} and some more advanced properties in \cite{AJP2}. Both papers are used intensively along the paper. Another important ingredient is the deformation theory of smooth morphisms as exposed in \cite{P}. A full-fledged theory of deformation is developed in \cite{P2}, but this generality is not needed in the present situation.  

It is worth remarking that decomposition up to $p$ uses essentially the results of the aforementioned papers but the extension of the result at the dimension $p$, requires the full machinery of Grothendieck duality for formal schemes \cite{AJL}. Moreover, Sastry's computation of the dualizing sheaf of a pseudo-proper smooth noetherian formal scheme \cite{sastry} is required to reach the general result.

Let us now describe the contents of the paper. An initial section recalls the basic definitions and notations that will be of use throughout the paper. In particular we recall the definition of the module of differentials and the associated De Rham complex. In the next section we discuss the basic properties of the Frobenius morphism both in absolute and relative version. It is noteworthy that the Frobenius morphism is an adic homeomorphism. Moreover we show that it is a finite locally free morphism.

In Section \ref{cartsec} we develop Cartier theory for noetherian formal schemes. Specifically, in Theorem \ref{teorisocar}  we establish an analogous to the Cartier isomorphism in $\sch$ \cite[(7.2)]{k} for relative differential forms associated to a smooth morphism of locally noetherian formal schemes of characteristic $p$. 

Once all this structure is up and running we prove the decomposition theorem. We fix $\FY$ a locally noetherian formal scheme of characteristic $p$ together with $\widetilde{\FY}$, a flat lifting  over $\ZZ/p^{2} \ZZ$. Let $f \colon  \FX \to \FY$ be a smooth morphism of locally noetherian formal schemes, let us consider its relative Frobenius mophism denoted by $F_{\FX/\FY} \colon \FX \to \FX^{(p)}$. It holds that any smooth lifting $\widetilde{\FX^{(p)}}$ of $\FX^{(p)}$ over $\widetilde{\FY}$ yields a  a decomposition of the complex $\tau^{<p}(F_{\FX/\FY\,*} \om^{\bullet}_{\FX/\FY})$ in $\D(\FX^{(p)})$. Moreover, much as in the case of schemes, the existence of a smooth lifting is equivalent to the existence of a decomposition of $\tau^{<p}(F_{\FX/\FY\,*} \om^{\bullet}_{\FX/\FY})$. The proof relies on the theory of (non necessarily adic) smooth morphisms of formal schemes, its basic deformation theory and the lifting of Frobenius morphisms. Of course, a global  lifting of Frobenius is not guaranteed to exist, but only local liftings. The corresponding local decompositions are glued by a procedure similar to the one employed in \cite{deill}.

Finally, in Section \ref{secatp} we extend this result to degree $p$. For $k$ a perfect field of characteristic $p$ and $\FX$ a formal scheme of topological dimension less or equal than $p$, we show that $F_{\FX/k\,*} \om^{\bullet}_{\FX/k}$ is decomposable. This is Theorem \ref{atp}. Its proof requires establishing a pairing on differential forms
\[
F_{\FX/k\,*}\om^{\, i}_{\FX/k} \otimes_{\CO_{\FX^{(p)}}} F_{\FX/k\,*}\om^{n-i}_{\FX/k} \lto \omega_{\FX^{(p)}/k}
\]
where $\omega_{\FX^{(p)}/k} = \om^{n}_{\FX^{(p)}/k}$, that is dualizing for coherent coefficients by Sastry's result \cite[Theorem 5.1.2]{sastry}. On formal schemes there are basically two dualities, one that refers to \emph{torsion} coefficients and another one for \emph{complete} coefficients ---this last one including the familiar coherent complexes. There is a balance between them controlled by Greenlees-May duality. It is this balance that provides an explicit description of the trace map as a Cartier operator, thereby allowing to extend Deligne-Illusie's idea to the present context.

In future work we will intend to apply the Decomposition Theorem to obtain vanishing theorems for formal schemes with an eye towards the cohomology of singular varieties. The main difficulty in this context is the lack of general finiteness properties. We expect to extend the available results in characteristic 0 to some situations in positive characteristic. With this in hand, the degeneration of the Hodge-De Rham spectral sequence would provide a path towards the desired results.

\begin{ack}
 We thank useful conversation and pointers to the literature to J. Lipman and K. Schwede. We give special thanks to P. Sastry for his interest on our work and the useful suggestions he has given us to improve this paper, especially the last section.
\end{ack}

\section{Preliminaries} \label{sec1}

We denote by $\sfn$ the category of  locally noetherian formal schemes and by $\sfna$ the subcategory of locally noetherian affine formal schemes. We follow the conventions and notations in \cite[\S 10]{EGA1}. Except otherwise indicated, every formal scheme will be in $\sfn$ and we will assume that every ring is noetherian. We write $\sch$ for the category of ordinary schemes.

\begin{cosa}\label{basic}
Given $\FX \in \sfn$ we denote by $\A(\FX)$ the category of $\CO_{\FX}$-Modules and $\D(\FX)$  its corresponding derived category. We denote by $\A_{\cc}(\FX) \subset \A(\FX)$ the full subcategory of coherent $\CO_{\FX}$-Modules\footnote{We honor the capitalization conventions in EGA and write ``Ideal'' and ``Module'' for sheaves of ideals and modules respectively.} and by $\D_{\cc}(\FX)$ the full subcategory of $\D(\FX)$ of complexes whose homology sheaves lie in $\A_{\cc}(\FX)$.

Given $f \colon \FX \to \FY$ a map of formal schemes, $f^{\sharp}\colon \CO_{\FY} \to f_* \CO_{\FX}$ will denote the corresponding morphism of structure sheaves and, with a slight abuse of notation, the ring homomorphisms it induces on sections and stalks.
\end{cosa}

\begin{cosa}  \label{lim}
Let us establish the following convenient notation (\cfr{}\cite[\S 10.6]{EGA1}):
\begin{enumerate}
\item
Given $\FX \in \sfn$ and  $\CJ \subset \CO_{\FX}$ an Ideal of definition,  for each $\ell\in \NN$ we put $X_{\ell}:=(\FX,\CO_{\FX}/\CJ^{\ell+1})$. In the category of formal schemes
\[\FX=\dirlim {\ell \in \NN} X_{\ell},\]
and all the spaces $X_{\ell}$  and $\FX$   have the same underlying topological space.

\item If $f\colon \FX \to \FY$ is a morphism in $\sfn$, given an Ideal of definition $\CK \subset \CO_{\FY}$ there exists an Ideal of definition $\CJ \subset \CO_{\FX}$ such that $f^*(\CK)\CO_{\FX} \subset \CJ $. For any such a pair of ideals setting $X_{\ell}:=(\FX,\CO_{\FX}/\CJ^{\ell+1})$ and $Y_{\ell}:=(\FY,\CO_{\FY}/\CK^{n+1})$ and $f_{\ell}\colon  X_{\ell} \to Y_{\ell}$ the scheme morphism induced by $f$ for each $\ell \in \NN$,  $f$ can be expressed as 
\[f = \dirlim {\ell \in \NN} f_{\ell}.\]
\end{enumerate}
\end{cosa}

\begin{cosa}
As in \cite[Definitions 1.6 and 1.7]{AJP2}, given $\FX \in \sfn$, the  \emph{topological dimension  of   $\FX$} is
$
\dimtop(\FX) := \dim (X_{0})
$
and the \emph{algebraic dimension of  $\FX$} is 
$
\dim(\FX) := \sup_{x \in \FX} \dim\CO_{\FX,x}
$. 
Obviously, 
\[ \dimtop (\FX)\le \dim (\FX).\] 

\end{cosa}

\begin{cosa}
Let us recall some definitions from  \cite[10.13.3]{EGA1}, \cite[(4.8.2)]{EGA31}, \cite[p.7]{AJL}, \cite[\S2 and \S3]{AJP1}. A morphism $f \colon  \FX \to \FY$ in $\sfn$ is of \emph{pseudo finite type (pseudo finite, pseudo proper, separated)} if $f_0$ (equivalently any $f_{\ell}$) is of finite type (finite, proper, separated, respectively). Moreover, we say that $f$ is of \emph{finite type (finite, proper)}  if $f$ is adic and of pseudo finite type (pseudo finite, pseudo proper, respectively).

The morphism $f$ is \emph{smooth  (unramified, \'etale)} if it is of pseudo finite type and satisfies the following lifting condition:
\emph{for any affine $\FY$-scheme $Z$ and for each closed subscheme $T\inc Z$ given by a square zero Ideal $\CI \subset \CO_{Z}$ the induced map
\begin{equation} 
\Hom_{\FY}(Z,\FX) \lto \Hom_{\FY}(T,\FX)
\end{equation}
is surjective (injective, bijective, respectively).}
\end{cosa}

\begin{cosa}

Given $f \colon \FX \to \FY$ a morphism in $\sfn$, for  all open sets $\FU = \spf(A) \subset \FX$ and $\FV=\spf(B) \subset \FY$ such that $f(\FU) \subset \FV$ the  \emph{differential pair of  $\FX$ over $\FY$}, $(\om^{1}_{\FX/\FY}, \hd_{\FX/\FY})$, is locally  given  by
\( ((\om^{1}_{A/B})^\tr,\hd_{A/B}^{\, \tr})
\) where $\tr$ \cite[(10.10.1)]{EGA1} is the
additive covariant  functor 
\begin{equation} \label{triangulito}
\begin{array}{ccc}
(-)^{\tr} \colon A\modu& \lto & \A(\spf(A))\\
M					&  \rightsquigarrow    &  M^{\tr}\\
\end{array}
\end{equation}
The  $\CO_{\FX}$-Module $\om^{1}_{\FX/\FY}$ is called the \emph{Module of  $1$-differentials of  $\FX$ over $\FY$} and the continuous $\FY$-derivation $\hd_{\FX/\FY}\colon \CO_{\FX} \to \om^{1}_{\FX/\FY}$ is called the \emph{canonical derivation of  $\FX$ over $\FY$}.

If we express as in \ref{lim} 
\[
f \colon  \FX \to \FY  \quad = \quad
 \dirlim {\ell\in \NN} (f_{\ell} \colon X_{\ell} \to Y_{\ell}),
\]
we have the following identification
\[
\CO_{\FX} \xto{\hd_{\FX/\FY}} \om^{1}_{\FX/\FY} \quad = \quad
\invlim {\ell \in \NN} (\CO_{X_{\ell}} \xto{d_{X_\ell/Y_\ell}}
{\Omega^{1}_{X_\ell/Y_\ell}}).
\]

From now on and whenever is clear, we will abbreviate $\hd=\hd_{\FX/\FY}$. 
\end{cosa}

\begin{cosa} \label{defidif}
For all $i \in \ZZ$, the \emph{sheaf of $i$-differentials of $\FX$ over $\FY$} is the sheaf $\omi_{\FX/\FY} := \pei \om^{1}_{\FX/\FY}$. Given open subsets $\FU=\spf(A) \subset \FX$ and $\FV=\spf(B) \subset \FY$ with $f(\FU) \subset \FV$, $\omi_{\FX/\FY}$ is locally  given  by
\( (\pei \om^{1}_{A/B})^{\tr} \) as a sheaf on $\FU \subset \FX$. 

Notice that $\om^{0}_{\FX/\FY}= \CO_{\FX}$ and $\omi_{\FX/\FY}= 0$, for all $i<0$.

 \label{lemderhamcoherente}
If $f$ is of pseudo finite type, then\footnote{If $f \colon  X \to Y$ is a morphism in $\sch$, then $\Omega^{i}_{X/Y}$ is a quasi-coherent  $\CO_{X}$-module. However, in the context of of formal schemes, to have a satisfactory description of the sheaf of $i$-differentials  we will restrict ourselves  to the class of morphisms of pseudo finite type.}
for all $i$, $\omi_{\FX/\FY} \in \A_{\cc}(\FX)$ (see  \cite[Proposition 2.6.1]{LNS} keeping in mind \cite[(10.10.2.9)]{EGA1}).
\end{cosa}

From now on $f$ will be a morphism of pseudo finite type.

\begin{cosa} \label{cbasepei}
We denote by  $\om^{\bullet}_{\FX/\FY}$ the sheaf of graded abelian groups that to an open subset $\FU \subset \FX$ associates the module
 \[
 \FU \leadsto \ga(\FU, \om^{\bullet}_{\FX/\FY}):= \bigoplus_{q \in \NN} \ga(\FU,\om^{q}_{\FX/\FY}) 
 \]
The sheaf $\om^{\bullet}_{\FX/\FY}$ is a supercommutative $\CO_{\FX}$-Algebra (\ie{}graded and alternating in the terminology of  \cite[Ch. III, \S7.1, Definition 1 and \S7.3, Proposition 5]{boual}).

For a  commutative diagram of morphisms in $\sfn$,
\begin{equation}\label{cbasepeidiag}
\begin{tikzpicture}[baseline=(current  bounding  box.center)]
\matrix(m)[matrix of math nodes, row sep=2.4em, column sep=2.4em,
text height=1.5ex, text depth=0.25ex]{
\FX'  & \FX \\
\FY'  & \FY \\}; 
\path[->,font=\scriptsize,>=angle 90]
(m-1-1) edge node[auto] {$g$} (m-1-2)
        edge node[left] {$f'$} (m-2-1)
(m-1-2) edge node[auto] {$f$} (m-2-2)
(m-2-1) edge node[auto] {$h$} (m-2-2);
\end{tikzpicture}
\end{equation}
such that $f$ and $f'$ are of pseudo finite type, there exists a morphism of graded $\CO_{\FX'}$-Algebras 
\begin{equation}\label{cbasepeimap}
 g^{*}\om^{\, \bullet}_{\FX/\FY} \lto \om^{\, \bullet}_{\FX'/\FY'}
\end{equation}
determined locallly in degree $i$ by  
\[
(\hd  a_{1} \wedge\hd  a_{2} \wedge \ldots \wedge \hd a_{i}) \otimes 1 \rightsquigarrow \hd'  g^{\sharp}(a_{1}) \wedge \hd' g^{\sharp}(a_{2}) \wedge \ldots \wedge \hd' g^{\sharp}(a_{i}) 
\] 
for any $a_{1},\, a_{2}, \ldots, a_{i} \in \ga(\FU,\CO_{\FX})$ with $\FU \subset \FX$ an affine open set (\cite[Proposition 3.7]{AJP1} and  \cite[Ch. III, \S7.1, Proposition 1]{boual}).
Moreover, if the diagram (\ref{cbasepeidiag}) is cartesian, the morphism (\ref{cbasepeimap}) is an isomorphism.
\end{cosa}

\begin{cosa}\label{existyunicomderham}
Analogously to the case of schemes (see \cite[(16.6.2)]{EGA44}), there exists an unique graded morphism of degree $1$ 
\[
\hd^{\, \bullet} \colon \om^{\bullet}_{\FX/\FY} \lto \om^{\bullet}_{\FX/\FY}
\]
such that:
\begin{enumerate}
\item
$\hd^{\, 0}= \hd$,
\item
$\hd^{\, i+1} \circ \hd^{\, i}=0$, for all $i \in \mathbb{N}$ and
\item
given  $\FU \subset \FX$ an open set,  $w_{i} \in  \ga(\FU, \om^{i}_{\FX/\FY})$ and $w_{j} \in  \ga(\FU, \om^{j}_{\FX/\FY})$, 
\[
\hd^{\, i+j}( w_{i}  \wedge w_{j})= \hd^{\, i}( w_{i}  )\wedge w_{j} + (-1)^{i}  w_{i} \wedge \hd^{\, j}(w_{j})
\]
for any $i,\, j \in \mathbb{N}$.
\end{enumerate}
Then
\[
(\om^{\bullet}_{\FX/\FY}, \hd^{\, \bullet})\colon  0 \to \CO_{\FX} \xto{\hd} \om^{1}_{\FX/\FY} \xto{\hd^{\, 1}} \cdots \xto{\hd^{\, k-1}} \om^{k}_{\FX/\FY} \xto{\hd^{\, k}} \cdots
\]
is a complex of coherent $\CO_{\FX}$-Modules; it is called \emph{De Rham complex of $\FX$ relative to $\FY$}. We abbreviate it by  $\om^{\bullet}_{\FX/\FY}$. Notice that the differentials are $f^{-1}\CO_{\FY}$-linear but not $\CO_{\FX}$-linear.

Observe that if  $f \colon  X \to Y$ is a finite type morphism of usual schemes then $\om^{\bullet}_{X/Y}=\Omega^{\bullet}_{X/Y}$.

 \label{cambiobasederham}
In the setting of the  commutative diagram (\ref{cbasepeidiag}),
the  morphism of graded  $\CO_{\FX}$-Algebras $\om^{\bullet}_{\FX/\FY} \to g_{*} \om^{\bullet}_{\FX'/\FY'}$ adjoint to (\ref{cbasepeimap}) respects the differential, \ie{}it is a map of complexes. 

\end{cosa}

\begin{cosa}  \label{prodextlocallibr}
Suppose that $f\colon \FX \to \FY$ is smooth and such that, for all $x \in \FX$, $\dim_{x} f := \dim f^{-1}(f(x))= n$ \cite[Definition 1.14]{AJP2}. Then $\om^{1}_{\FX/\FY}$ is  a locally free $\CO_{\FX}$-Module of rank $n$ (see  \cite[Proposition 2.6.1]{LNS} and \cite[Corollary 5.10]{AJP2}) and therefore
 $\omi_{\FX/\FY}$ is  a locally free $\CO_{\FX}$-Module of constant rank
$ \binom {n}{i},\, \textrm{ for all }0 \le i \le n$.
In particular, $\om^{n}_{\FX/\FY}$ is an invertible $\CO_{\FX}$-Module and $\omi_{\FX/\FY} =0$, for all $i > n$.
Therefore $\om^{\bullet}_{\FX/\FY}$ is a bounded complex  of amplitude $[0,n]$ of locally free $\CO_{\FX}$-Modules.
\end{cosa}

\begin{rem}
 Let $f \colon X \to \spec (\mathbb{C})$ be a smooth morphism of usual schemes, $Z\subset X$ a closed subscheme and denote by  $\widehat{X}$ the completion of  $X$ along $Z$.  The De Rham  complex of $\widehat{X}$ relative to $ \mathbb{C}$ defined above, $\om^{\bullet}_{\widehat{X}/\mathbb{C}}$,  agrees with the one given by  Hartshorne in \cite[I, \S 7]{ha2}.
\end{rem}

\section{Frobenius morphism on formal schemes} \label{sec2}

Henceforth, $p$ will denote a prime number and $\mathbb{F}_{p} := \mathbb{Z}/ p \mathbb{Z}$ the prime field.

\begin{cosa} \label{principiofrob}
A locally noetherian formal scheme $\FX$  is of  \emph{characteristic $p$} if the canonical morphism $\FX \to \spec(\mathbb{Z})$ factors through $\spec(\mathbb{F}_{p})$, that is, if $p \cdot \CO_{\FX}=0$. Equivalently, given an ideal of definition $\CJ \subset \CO_{\FX}$, the schemes $X_{\ell}=(\FX,\CO_{\FX}/\CJ^{\ell+1})$ are of characteristic $p$, for all $\ell \in \NN$.
\end{cosa}

\begin{cosa} \label{defnfrobabs}
Let $\FX$ be a locally noetherian formal scheme of characteristic $p$. The \emph{absolute Frobenius endomorphism of  $\FX$}, is the endomorphism $F_{\FX}\colon  \FX \to \FX$ that is the identity as a map of topological spaces and, given for all open set $\FU \subset \FX$ by 
\[
\begin{array}{ccc}
\ga(\FU,\CO_{\FX}) &\xto{\ga(\FU, F_{\FX})}& \ga(\FU,\CO_{\FX})\\ 
a&\leadsto & a^{p}
\end{array}
\]
The following holds:
\begin{enumerate}
\item \label{frobabsadico}
The morphism $F_\FX$ is adic. Indeed, for a noetherian adic ring $A$ \cite[(10.4.6)]{EGA1}, $J \subset A$ an ideal of definition, and $F_A \colon A \to A$ its Frobenius endomorphism, the ideal $J^{e}=\langle F_{A}(J) \rangle$ defines the $J$-adic topology in $A$.
\item
\label{absfroblimabsfro}
Given an Ideal of definition $\CJ \subset \CO_{\FX}$ if $F_{X_{\ell}}\colon  X_{\ell} \to X_{\ell}$ is the absolute Frobenius endomorphism of $X_{\ell}$, for all $\ell \in \NN$, then 
\[
F_{\FX} = \dirlim {\ell \in \NN} F_{X_{\ell}}.
\] 
\item \label{frobabshomeo}
$F_\FX$ is  a \emph{universal homeomorphism}, that is, a homeomorphism such that for each morphism of locally noetherian formal schemes $\FZ \to \FX$, the morphism obtained by  base-change $\FX \times \FZ  \to \FZ$ is a homeomorphism. Indeed, with the previous notation,
as $F_{X_{\ell}}$ is a universal  homeomorphism (see \cite[Expos\'e XV, \S1]{sga5}) in view of \cite[(10.7.4)]{EGA1}  we deduce that $F_{\FX}$ is too, because $(F_{\FX})_{\textrm{top}} = (F_{X_{\ell}})_{\textrm{top}}$.
\end{enumerate}
\end{cosa}

\begin{cosa} \label{defmorfrobrel}
For $f \colon  \FX \to \FY$ in $\sfn$ with $\FY$   of characteristic $p$, we have the following commutative diagram
\begin{equation*}
\begin{tikzpicture}[baseline=(current  bounding  box.center)]
\matrix(m)[matrix of math nodes, row sep=2.4em, column sep=2.4em,
text height=1.5ex, text depth=0.25ex]{
\FX  & \FX \\
\FY  & \FY \\}; 
\path[->,font=\scriptsize,>=angle 90]
(m-1-1) edge node[auto] {$F_{\FX}$} (m-1-2)
        edge node[left] {$f$} (m-2-1)
(m-1-2) edge node[auto] {$f$} (m-2-2)
(m-2-1) edge node[auto] {$F_{\FY}$} (m-2-2);
\end{tikzpicture}
\end{equation*}
where the horizontal arrows are the absolute Frobenius endomorphisms  of  $\FX$ and $\FY$.
Let us  put  $\FX^{(p)} := \FX \times_{F_{\FY}} \FY$. Notice the dependence of the formal scheme $\FX^{(p)}$ on the base $\FY$. We omit it on the notation for clarity. There exists an unique morphism 
\[
F_{\FX/\FY} \colon \FX \lto \FX^{(p)}
\]
 that makes commutative the diagram
\begin{equation}\label{diagrfrob}
\begin{tikzpicture}[baseline=(current  bounding  box.center)]
\matrix(m)[matrix of math nodes, row sep=2.8em, column sep=3.2em,
text height=1.5ex, text depth=0.25ex]{
\FX &       &     \\
    & \FX^{(p)}  & \FX \\
    & \FY   & \FY \\}; 
\path[->,font=\scriptsize,>=angle 90]
(m-1-1) edge node[right] {$F_{\FX/\FY}$} (m-2-2)
        edge[bend left=20] node[above] {$F_{\FX}$} (m-2-3)
        edge[bend right=20] node[left] {$f$} (m-3-2)
(m-2-2) edge node[above] {$(F_{\FY})_{\FX}$} (m-2-3)
        edge node[left] {$f^{(p)}$} (m-3-2)
(m-2-3) edge node[auto] {$f$} (m-3-3)
(m-3-2) edge node[auto] {$F_{\FY}$} (m-3-3);
\end{tikzpicture}
\end{equation}
The morphism $F_{\FX/\FY}$  is called  \emph{relative Frobenius  morphism of  $\FX$ over $\FY$}.

Given Ideals of definition $\CJ \subset \CO_{\FX}$ and $\CK \subset \CO_{\FY}$ such that $f^{*}(\CK) \CO_{\FX} \subset \CJ$, if $F_{X_{\ell}/Y_{\ell}}\colon  X_{\ell} \to X^{(p)}_{\ell}$ is the relative Frobenius morphism from $X_{\ell}$ to $Y_{\ell}$, by \ref{defnfrobabs}.(\ref{absfroblimabsfro}) and \cite[(10.7.4)]{EGA1} we have that 
\[
\begin{array}{ccc}
\begin{tikzpicture}[baseline=(current  bounding  box.center)]
\matrix(m)[matrix of math nodes, row sep=2.8em, column sep=3.2em,
text height=1.5ex, text depth=0.25ex]{
\FX &       &     \\
    & \FX^{(p)}  & \FX \\
    & \FY   & \FY \\}; 
\path[->,font=\scriptsize,>=angle 90]
(m-1-1) edge node[right] {$F_{\FX/\FY}$} (m-2-2)
        edge[bend left=20] node[above] {$F_{\FX}$} (m-2-3)
        edge[bend right=20] node[left] {$f$} (m-3-2)
(m-2-2) edge node[above] {$(F_{\FY})_{\FX}$} (m-2-3)
        edge node[left] {$f^{(p)}$} (m-3-2)
(m-2-3) edge node[auto] {$f$} (m-3-3)
(m-3-2) edge node[auto] {$F_{\FY}$} (m-3-3);
\end{tikzpicture}
&
= \dirlim {n \in \NN} 
& 
  \left(
\begin{tikzpicture}[baseline=(current  bounding  box.center)]
\matrix(m)[matrix of math nodes, row sep=2.8em, column sep=3.2em,
text height=1.5ex, text depth=0.25ex]{
X_{\ell} &         &     \\
      & X^{(p)}_{\ell}  & X_{\ell} \\
      & Y_{\ell}   & Y_{\ell} \\}; 
\path[->,font=\scriptsize,>=angle 90]
(m-1-1) edge node[right] {$F_{X_{\ell}/Y_{\ell}}$} (m-2-2)
        edge[bend left] node[above] {$F_{X_\ell}$} (m-2-3)
        edge[bend right] node[left] {$f_{\ell}$} (m-3-2)
(m-2-2) edge node[above] {$(F_{Y_{\ell}})_{X_{\ell}}$} (m-2-3)
        edge node[left] {$f^{(p)}_\ell$} (m-3-2)
(m-2-3) edge node[auto] {$f_{\ell}$} (m-3-3)
(m-3-2) edge node[auto] {$F_{Y_\ell}$} (m-3-3);
\end{tikzpicture}
  \right)
\end{array}
\]
and, in particular, $F_{\FX/\FY} = \dirlim {\ell \in \NN} F_{X_{\ell}/Y_{\ell}}$.
 
\end{cosa}

\begin{cosa} \label{frobrelafin}
Let $\varphi \colon A \to B$ be a homomorphism of noetherian adic rings of characteristic $p$; let $\FX= \spf(A)$, $\FY= \spf(B)$ and $f  \colon  \FX \to \FY$  such that $f := \spf({\varphi})$ is in $\sfna$.
The diagram (\ref{diagrfrob}) corresponds through the equivalence of categories to the following diagram 
\begin{equation*}
\begin{tikzpicture}[baseline=(current  bounding  box.center)]
\matrix(m)[matrix of math nodes, row sep=3.2em, column sep=3.6em,
text height=1.5ex, text depth=0.25ex]{
B & B                &  \\
A & A \tc_{F_{B}} B  &  \\
    &                & A \\}; 
\path[->,font=\scriptsize,>=angle 90]
(m-1-1) edge node[above] {$F_{B}$} (m-1-2)
        edge node[left]  {$\varphi$} (m-2-1)
(m-1-2) edge node[right] {$\varphi'$} (m-2-2)
        edge[bend left=20]  node[right] {$\varphi$} (m-3-3)
(m-2-1) edge node[above] {$(F_{B})_{\!A}$} (m-2-2)
        edge[bend right=20] node[left] {$F_{A}\quad$} (m-3-3)
(m-2-2) edge node[above] {$\quad F_{A/B}$} (m-3-3);
\end{tikzpicture}
\end{equation*}
where  $F_A$ are $F_B$ are the usual Frobenius homomorphisms (raise to the $p$-th power), $F_{A/B}(a\tc b) = a^{p} \cdot \varphi(b)$, denoting by $a\tc b \in A \tc_{F_{B}} B$  the image of $a \otimes b \in A \otimes_{F_{B}} B$ and $(F_{B})_{\!A}(a)=a\tc 1$.
\end{cosa}

\begin{prop} \label{propifrobreladicohomeo} 
Given $f \colon  \FX \to \FY$ in $\sfn$ with $\FY$  of characteristic $p$ and $F_{\FX/\FY}$  the relative Frobenius morphism of $\FX$ over $\FY$ it holds that:
\begin{enumerate}
\item
$F_{\FX/\FY}$ is adic.
\item
$F_{\FX/\FY}$ is  a \emph{homeomorphism}. 
\end{enumerate}
\end{prop}

\begin{proof}
Let us consider the diagram (\ref{diagrfrob}).
\begin{enumerate}
  \item The morphisms $F_{\FX}$ and $F_{\FY}$  are adic (\ref{defnfrobabs}.(\ref{frobabsadico})). By  base-change (see \cite[1.3]{AJP1}), we have that the morphism $(F_{\FY})_{\FX}$ is adic. Therefore $F_{\FX/\FY}$ also is adic (see \cite[(10.12.1)]{EGA1}).
  \item It follows from  \ref{defnfrobabs}.(\ref{frobabshomeo}).
\qedhere
\end{enumerate}
\end{proof}

\begin{cosa} \label{frobespacafin}
Given $\FY=\spf(B)$ a  noetherian affine formal scheme of characteristic $p$, $n>0$ and  $\pi: \AF^{n}_{\FY}=\spf(B\{\mathbf{T}\}) \to \FY$ the canonical projection of the affine formal space, it holds that:
\begin{enumerate}
\item \label{frobespacafiniso}
There exists an isomorphism of $\FY$-formal schemes
\[
(\AF^{n}_{\FY})^{(p)}= \AF^{n}_{\FY} \times_{F_{\FY}} \FY \xto{\overset{\Phi}\sim}  \AF^{n}_{\FY}\]
defined through the equivalence of categories by  the morphism
\[
\begin{array}{ccc}
B\{\mathbf{T}\} &\xto{\quad\Phi\quad} &B\{\mathbf{T}\}  \tc_{F_{B}} B\\
\sum_{\nu \in \NN^{n}} b_{\nu} \mathbf{T}^{\nu}&\leadsto & \sum_{\nu \in \NN^{n}} \mathbf{T}^{\nu} \tc b_{\nu}
\end{array}
\]
given  by  the universal property of the restricted formal power series ring (\emph{cf.}  \cite[Ch. III, \S 4.2, Proposition 4]{b}). Let us check that $\Phi$ is an isomorphism. If  $B\{\mathbf{T}\} \xto{G}  B\{\mathbf{T}\}$ is the morphism induced by  $F_{B}$, applying the universal property of the complete tensor product there exists an unique morphism $\Psi: B\{\mathbf{T}\} \tc_{F_{B}}B\to B\{\mathbf{T}\}$ such that the following diagram  commutes:
\begin{equation*}
\begin{tikzpicture}[baseline=(current  bounding  box.center)]
\matrix(m)[matrix of math nodes, row sep=3.2em, column sep=3.6em,
text height=1.5ex, text depth=0.25ex]{
B                & B                             &  \\
B\{\mathbf{T}\}	 & B\{\mathbf{T}\}  \tc_{F_{B}} B&  \\
                 &                               & B\{\mathbf{T}\} \\}; 
\path[->,font=\scriptsize,>=angle 90]
(m-1-1) edge node[above] {$F_{B}$} (m-1-2)
        edge node[left]  {$\pi$} (m-2-1)
(m-1-2) edge node[right] {$\pi'$} (m-2-2)
        edge[bend left=20]  node[right] {$\pi$} (m-3-3)
(m-2-1) edge node[above] {$(F_{B})_{B\{\mathbf{T}\}}$} (m-2-2)
        edge[bend right=20] node[left] {$G\quad$} (m-3-3)
(m-2-2) edge node[above] {$\Psi$} (m-3-3);
\end{tikzpicture}\nu
\end{equation*}
Therefore $\Psi(\sum_{\nu \in \NN^{n}} b_{\nu} \mathbf{T}^{\nu} \tc b) = \sum_{\nu \in \NN^{n}} b\cdot b_{\nu}^{p}  \mathbf{T}^{\nu}$ and $\Phi^{-1}= \Psi$.
\item \label{frobseries}
The morphisms $F_{\AF^{n}_{\FY}/\FY}$ and $(F_{\FY})_{\AF^{n}_{\FY}}$ are determined
by:
\[
\begin{array}{ccc}
B\{\mathbf{T}\}&\xto{\quad\sigma\quad}& B\{\mathbf{T}\}\\ 
\sum_{\nu \in \NN^{n}} b_{\nu} \mathbf{T}^{\nu}&\leadsto & \sum_{\nu \in \NN^{n}} b_{\nu} (\mathbf{T}^{\nu}) ^{p}
\end{array}
\]
and
\[ 
\begin{array}{ccc}
B\{\mathbf{T}\}&\xto{\quad\tau \quad} & B\{\mathbf{T}\}\\ 
\sum_{\nu \in \NN^{n}} b_{\nu} \mathbf{T}^{\nu}&\leadsto & \sum_{\nu \in \NN^{n}} b_{\nu}^{p} \mathbf{T}^{\nu}
\end{array}
\]
through the isomorphisms $\Psi$ and $\Phi$, respectively.

\item\label{frobrelproyecfinitflat}
The relative Frobenius  morphism of $\AF_{\FY}^{n}$ over $\FY$,  $F\colon  \AF_{\FY}^{n} \to (\AF_{\FY}^{n})^{(p)}$,  is  finite, flat and $F_{*}(\CO_{\AF^{n}_{\FY}})$ is a locally free $\CO_{(\AF^{n}_{\FY})^{(p)}}$-Algebra of rank $p^{n}$. In fact,  through the morphism $\sigma$, $B\{\mathbf{T}\}$ is a free $B\{\mathbf{T}\}$-module with base $\{\prod_{i=1}^{n}T_{i}^{m_{i}}, 0 \le m_{i} \le p-1\}$. 
\end{enumerate}

\end{cosa}
\vspace{20pt}

If $f \colon  X \to Y$ is an \'etale morphism of locally noetherian schemes  of characteristic $p$, then the relative Frobenius morphism of $X$ over $Y$ is an isomorphism \cite[Expose XV, \S1]{sga5}. Next we generalize this result to the setting of locally noetherian formal schemes.

\begin{lem} \label{lemetfrobresmooth}
Given  a locally noetherian formal scheme $\FY$ of characteristic $p$, let $f \colon  \FX \to \FY$ be an \'etale morphism in $\sfn$. Then the relative Frobenius morphism of $\FX$ over $\FY$, $F_{\FX/\FY}\colon  \FX \to \FX^{(p)}= \FX \times_{F_{\FY}} \FY$, is an isomorphism.
\end{lem}

\begin{proof}
Let us consider the  commutative diagram (\ref{diagrfrob}).
The morphism $f$ is \'etale and by  base-change (see \cite[Proposition 2.9, (ii)]{AJP1}) it follows that $f'$ is \'etale. Then \cite[Corollary 2.14]{AJP1} and Proposition \ref{propifrobreladicohomeo} imply that $F_{\FX/\FY}$ is \'etale adic. On the other hand,  by  \ref{defnfrobabs}.(\ref{frobabshomeo}), $F_{\FX}$ is a universal homeomorphism and, therefore, radical (see \cite[Definition 2.5]{AJP2}). From the sorites of radical morphisms in $\sch$  \cite[Corollaire (3.7.6)]{EGA1} we have that $F_{\FX/\FY}$ is a radical  morphism and applying \cite[Theorem 7.3]{AJP2}   it follows that $F_{\FX/\FY}$ is an open inmersion. Last, by  Proposition \ref{propifrobreladicohomeo} we have that $F_{\FX/\FY}$ is a homeomorphism, so we conclude that it is an isomorphism.
\end{proof}

\begin{rem}
The last result does not follow straightforward from the analogous result in the category of schemes, since given 
\[
f= \dirlim {\ell \in \NN} f_{\ell}
\]
 an \'etale morphism of locally noetherian formal schemes it may happen that the corresponding morphisms of schemes $f_{\ell}$ in the system are not \'etale (see \cite[Example 5.3]{AJP2}).
\end{rem}

In Proposition \ref{frobfinitplan} we generalize  \ref{frobespacafin}.(\ref{frobrelproyecfinitflat}) for smooth morphisms   of locally noetherian formal schemes of constant relative dimension equal to $n$. First, we need a  previous result.

\begin{prop} \label{cambiobasefinito}
Given a cartesian diagram in $\sfn$
\begin{equation*}
\begin{tikzpicture}[baseline=(current  bounding  box.center)]
\matrix(m)[matrix of math nodes, row sep=3em, column sep=3em,
text height=1.5ex, text depth=0.25ex]{
\FX'  & \FY' \\
\FX   & \FY \\}; 
\path[->,font=\scriptsize,>=angle 90]
(m-1-1) edge node[auto] {$f'$} (m-1-2)
        edge node[left] {$g'$} (m-2-1)
(m-1-2) edge node[auto] {$g$} (m-2-2)
(m-2-1) edge node[auto] {$f$} (m-2-2);
\end{tikzpicture}
\end{equation*}
with $f$ finite, if $\CF \in \A_{\cc}(\FX)$ then the canonical morphism of $\CO_{\FY'}$-Modules
\begin{equation} \label{cambiodebaseafin1}
f'_{*}g'^{*}\CF \to g^{*}f_{*}\CF    
\end{equation}
is an isomorphism.
\end{prop}

\begin{proof}
By base-change  we have that $f'$  is also a  finite morphism (see \cite[Proposition 7.1]{AJL}). Then by the Finiteness Theorem  for finite  morphisms  in $\sfn$ \cite[(4.8.6)]{EGA31} it follows that $f'_{*}(g'^{*}\CF)$ and  $g^{*}(f_{*}\CF) $ are coherent $\CO_{\FY'}$-Modules.
Since this is a local question on the base, we may suppose that $g \colon  \FY'=\spf(B')\to \FY=\spf(B)$ is affine, and that $g'\colon  \FX' =\spf(A')\to \FX=\spf(A)$ is a morphism of affine formal schemes  with  $A$ a $B$-module of finite type and $A'= B'\otimes_{B} A$ a $B'$-module of finite type.
Applying the category equivalence given by  the  functor $(-)^\tr$ (see \cite[(10.10.5)]{EGA1}) we get that there exists a  finitely generated $A$-module  $M$ such that $\CF=M^{\tr}$, with $M$ a $B$-module of finite type. The morphism (\ref{cambiodebaseafin1}) corresponds to the canonical isomorphism  of finitely generated $B'$-modules $A' \otimes_{A} M  \to B' \otimes_{B} M$.
\end{proof}

\begin{prop} \label{frobfinitplan}
Given  a locally noetherian formal scheme $\FY$ of characteristic  $p$, let $f \colon  \FX \to \FY$ be a smooth morphism  of relative dimension $n$. Then the relative Frobenius endomorphism  of $\FX$ over $\FY$, $F_{\FX/\FY}\colon \FX \to \FX^{(p)}$, is finite, flat and $F_{\FX/\FY\,*} \CO_{\FX}$ is a locally free $\CO_{\FX^{(p)}}$-Algebra of rank $p^{n}$.
\end{prop}

\begin{proof}
By \cite[Proposition 5.9]{AJP2} we have
 that  for each  $x \in \FX$,  there exists an open subset  $\FU \subset \FX$ with $x \in \FU$ such that $f|_{\FU}$ factors as 
 \[
\FU \xto{\,g\,} \AF^{n}_{\FY} \xto{\,\pi\,} \FY
 \]
 where $g$ is \'etale, $\pi$ is the canonical  projection and $n = \rg (\om^{1}_{\CO_{\FX,x}/\CO_{\FY,f(x)}})$. We may assume that $\FU = \FX$. Taking the  diagram (\ref{diagrfrob})  for the morphisms $g$, $\pi$ and $f$ we have the following commutative diagram of locally noetherian formal schemes
\begin{equation}\label{prisma}
\begin{tikzpicture}[cross line/.style={preaction={draw=white, -, line width=8pt}},
                    baseline=(current  bounding  box.center)]
\matrix(m)[matrix of math nodes, row sep=2em, column sep=4em,
text height=1.5ex, text depth=0.25ex]{
\FX            &               & \FX           \\
               & \FX^{(p)}         &               \\
\AF^{n}_{\FY}  &               & \AF^{n}_{\FY} \\
               & (\AF^{n}_{\FY})^{ (p)} &               \\
\FY            &               & \FY           \\
               & \FY           &               \\}; 
\path[->,font=\scriptsize,>=angle 90]
(m-1-1) edge node[auto] {$F_{\FX}$} (m-1-3)
        edge node[above] {$F_{\FX/\FY}$} (m-2-2)
        edge node[left] {$g$} (m-3-1)
(m-1-3) edge node[auto] {$g$} (m-3-3)
(m-3-1) edge node[near start, auto] {$F_{\AF^{n}_{\FY}}$} (m-3-3)
        edge node[below] {$\,F_{\AF^{n}_{\FY}/\FY}\,\,$} (m-4-2)
        edge node[left] {$\pi$} (m-5-1)
(m-2-2) edge node[above] {$(F_{\FX})_{\FY}\qquad$} (m-1-3)
        edge [cross line] node[near start, left] {$g'$} (m-4-2)
(m-5-1) edge node[near start, auto] {$F_{\FY}$} (m-5-3)
        edge node[below] {$\id_{\FY}$} (m-6-2)
(m-4-2) edge node[below] {$\qquad (F_{\FY})_{\AF^{n}_{\FY}}$} (m-3-3)
        edge [cross line] node[near start, left] {$\pi^{(p)}$} (m-6-2)
(m-3-3) edge node[auto] {$\pi$} (m-5-3)
(m-6-2) edge node[below] {$F_{\FY}$} (m-5-3);
   \node at ([shift=({0em,2.3em})]m-2-2) {$\square_{1}$};
   \node at ([shift=({2.5em,-4.2em})]m-1-1) {$\Diamond_{1}$};
   \node at ([shift=({2.5em,-0.8em})]m-2-2) {$\Diamond_{2}$};
   \node at ([shift=({2.5em,-0.8em})]m-4-2) {$\Diamond_{3}$};
\end{tikzpicture}
\end{equation}
where:
\begin{itemize}
\item
the horizontal arrows are the absolute Frobenius  endomorphisms of  $\FX,\, \AF^{n}_{\FY}$ and $\FY$;
\item
$\FX^{(p)} = \FX \times_{F_{\FY}} \FY$ and $\Diamond_{3}$ is a cartesian square (so  $\Diamond_{2}$ is a cartesian square, too).
\end{itemize} 
Since $g$ is \'etale, by  Lemma \ref{lemetfrobresmooth} we have that $\square_{1}$ is a cartesian square  and, since $\Diamond_{2}$ is a cartesian square we  deduce
that $\Diamond_{1}$ is a cartesian square. On the other hand,  in \ref{frobespacafin}.(\ref{frobrelproyecfinitflat}) we have shown that $F_{\AF^{n}_{\FY}/\FY}$ is finite, flat and that  $(F_{\AF^{n}_{\FY}/\FY})_{*}  \CO_{\AF^{n}_{\FY}}$ is a locally free $\CO_{(\AF^{n}_{\FY})^{(p)}}$-Algebra  of rank $p^{n}$ with base $\{ \prod_{i=1}^{n}T_{i}^{m_{i}}, 0 \le m_{i} \le p-1\}$. Then by  base-change (see \cite[Proposition 7.1]{AJL}) we have that $F_{\FX/\FY}$ is finite and flat. Moreover, from Proposition \ref{cambiobasefinito} it results that:
\[
F_{\FX/\FY\,*} \CO_{\FX} = F_{\FX/\FY\,*} g^{*} \CO_{\AF^{n}_{\FY}} = g'^{*}F_{\AF^{n}_{\FY}/\FY\,*} \CO_{\AF^{n}_{\FY}}
\]
and, therefore, by \ref{frobespacafin}(\ref{frobrelproyecfinitflat}), $F_{\FX/\FY\,*} \CO_{\FX}$ is a locally free  $\CO_{\FX^{(p)}}$-Algebra  of rank $p^{n}$ with base  $\{g'^{\sharp} ( \prod_{i=1}^{n}T_{i}^{m_{i}}), 0 \le m_{i} \le p-1\}$.
\end{proof}

\begin{cor} \label{corfrobfinitplan}
Let $\FY$ be a locally noetherian formal scheme of characteristic $p$ and $f \colon  \FX \to \FY$ be a smooth morphism of relative dimension $n$. Then $F_{\FX/\FY\,*} \omi_{\FX/\FY}$ is a locally free $\CO_{\FX^{(p)}}$-Module  of rank $p^{n}\cdot \binom{n}{i}$, for all $i \in \{0,1,\ldots, n \}$.
\end{cor}

\begin{proof}
Let $0 \le i \le n$. From \ref{prodextlocallibr} we have that $\omi_{\FX/\FY}$ is a locally free $\CO_{\FX}$-Module of rank $\binom{n} {i}$ and therefore, the result is consequence of Proposition \ref{frobfinitplan}.
\end{proof}


\section{Cartier isomorphism}\label{cartsec}

One of the technical tools more used for the differential study  of schemes of characteristic $p$ is the Cartier isomorphism \cite{ca}. Our next task will be to extend it to smooth morphisms of locally noetherian formal schemes of characteristic $p$ following \cite[(7.2)]{k}.

\begin{cosa} \label{relativonulo}
Given $\FY$ a  locally noetherian formal scheme of characteristic $p$  let $f \colon  \FX \to \FY$ be a morphism  of locally noetherian formal schemes. For all open set  $\FU \subset \FX$ and for all $a \in \ga(\FU,\CO_{\FX})$, it holds that 
\[
\hd(a^{p})= p\cdot a^{p-1} \cdot \hd (a) = 0
\]
Therefore  the absolute Frobenius morphism  of $\FX$ and the relative Frobenius  morphism   of $\FX$ over $\FY$ induce zero morphisms
\[
F_{\FX}^{*} \om^{1}_{\FX/\FY} \overset{0}\lto \om^{1}_{\FX/\FY}, \qquad
F_{\FX/\FY}^{*} \om^{1}_{\FX^{(p)}/\FY} 		       \overset{0}\lto \om^{1}_{\FX/\FY}
\]
respectively (see \ref{frobrelafin}). After all, the differentials are null being radical morphisms.
\end{cosa}

\begin{cosa}
Given a locally noetherian formal scheme $\FY$ of characteristic $p$ and $f \colon  \FX \to \FY$ in $\sfn$ it holds that
$F_{\FX/\FY\,*} \om^{\bullet}_{\FX/\FY}:=(F_{\FX/\FY\,*} \om^{\bullet}_{\FX/\FY}, F_{\FX/\FY\,*} \hd^{\, \bullet})$ is a complex of $\CO_{\FX^{(p)}}$-Modules. Indeed, given an open set $\FU \subset \FX$,  
$a \tc b \in \ga(\FU, \CO_{\FX^{(p)}})$ and $c \in \ga(\FU, F_{\FX/\FY\,*} \CO_{\FX})$ there results that:
\begin{align*}
F_{\FX/\FY\,*} \hd(a \tc b \cdot c) &= \hd ( F_{A/B}(a \tc b) \cdot c) 
          = \hd ( a^{p} \cdot b \cdot  c) 
          = b \cdot \hd ( a^{p}\cdot  c) =\\    
        & = b \cdot p \cdot a^{p-1} \cdot \hd (a)\cdot  c + b \cdot a^{p} \cdot \hd (c) =\\
        & = a \tc b \cdot F_{\FX/\FY\,*} \hd(c).
\end{align*}

It holds that the sheaves of abelian groups  $\bigoplus_{i \in \ZZ} \CZ^{i} (F_{\FX/\FY\,*}\om^{\bullet}_{\FX/\FY})$ and $\bigoplus_{i \in \ZZ} \CH^{i} (F_{\FX/\FY\,*}\om^{\bullet}_{\FX/\FY})$  have structure  of supercommutative $\CO_{\FX^{(p)}}$-Algebras  determined by  the exterior product  so, the elements of degree $1$ are of zero square.

\end{cosa}

\begin{cosa} \label{ImdirfrobcomRhamcoh}
Let $f \colon  \FX \to \FY$  be a smooth morphism  of locally noetherian formal schemes  of characteristic $p$.  In this setting, there exists   a unique morphism of graded $\CO_{\FX^{(p)}}$-Algebras
\begin{equation} \label{ecisocart}
\gamma \colon\bigoplus_{i \in \ZZ} \om^{i}_{\FX^{(p)}/\FY} \lto 
\bigoplus_{i \in \ZZ} \CH^{i} (F_{\FX/\FY\,*}\om^{\bullet}_{\FX/\FY})
\end{equation} 
such that $\gamma^{0}$ is the canonical  morphism $\CO_{\FX^{(p)}} \to F_{\FX/\FY\,*}\CO_{\FX}$ and $\gamma^{1}$ is locally given by  $\hd(a) \otimes 1 \leadsto[a^{p-1} \hd (a)]$.

Uniqueness follows from the fact that $\bigoplus_{i \in \ZZ} \CH^{i} (F_{ *}\om^{\bullet}_{\FX/\FY})$ is a graded $\CO_{\FX^{(p)}}$-Algebra where the elements  of degree $1$ are of square zero (\cfr{}\cite[Ch. III, \S7.1, Proposition 1]{boual}).

For the existence, 
applying \cite[\emph{loc. cit.}]{boual} it suffices to give $\gamma^{0}$ and $\gamma^{1}$ as above. Consider  the morphism $D$ defined, for every open set $\FU \subset \FX$, by  
\[
\begin{array}{cccc}
\ga(\FU,D) \colon &\ga(\FU,\CO_{\FX^{(p)}}) &\lto &\ga(\FU,\CH^{1} (F_{\FX/\FY\,*}\om^{\bullet}_{\FX/\FY}))\\
&a \tc 1		& \leadsto		& [a^{p-1} \hd (a)]
\end{array}
\]
It is well defined since: 
\begin{equation*}
F_{\FX/\FY\,*} \hd(a^{p-1} \hd (a))=  \hd(a^{p-1}) \wedge \hd (a)= (p-1)a^{p-2} \hd(a) \wedge \hd (a)=0
\end{equation*}
and, therefore, $a^{p-1} \hd (a) \in \ga(\FU, \CZ^{1} (F_{\FX/\FY\,*}\om^{\bullet}_{\FX/\FY}))$.

Let us show that $D \in \Dercont_{\FY}(\CO_{\FX^{(p)}}, \CH^{1} (F_{\FX/\FY\,*}\om^{\bullet}_{\FX/\FY}))$. It is easily checked that $D$ is a continuous morphism. First, we will prove that  is a morphism of sheaves of abelian groups. We take $\FU \subset \FX$ an open subset and $a_{1},\, a_{2} \in \ga(\FU, \CO_{\FX})$. 
Applying  formally $\hd$ to the equality
\[
(a_{1}+a_{2})^{p} =
a_{1}^{p}+a_{2}^{p}+p \cdot\left(\sum_{i=1}^{p-1} \frac{(p-1)!}{i! \cdot(p-i)!} \cdot a_{1}^{i} \cdot a_{2}^{p-i}\right)
\]
we deduce that
\begin{multline*}
p \cdot (a_{1}+a_{2})^{p-1} \hd(a_{1}+a_{2})= \\ p\cdot \left( a_{1}^{p-1} \hd(a_{1}) +  a_{2}^{p-1} \hd(a_{2}) + \hd\left(\sum_{i=1}^{p-1} \frac{(p-1)!}{i! \cdot(p-i)!} \cdot a_{1}^{i} \cdot a_{2}^{p-i}\right)\right)
\end{multline*}
from which it follows that $D((a_{1}+ a_{2}) \tc 1) = D(a_{1} \tc 1) + D(a_{2} \tc  1)$.  
Next,
\[
\begin{split}
D((a_{1}\cdot a_{2}) \tc 1)&= [(a_{1} \cdot a_{2})^{p-1}\hd(a_{1} \cdot a_{2})]\\&=  [a_{2}^{p}\cdot a_{1}^{p-1}\hd(a_{1})]+ [a_{1}^{p}\cdot a_{2}^{p-1} \hd(a_{2}))]\\&= (a_{2} \tc  1)\cdot D(a_{1} \tc 1)+ (a_{1}\tc 1) \cdot D(a_{2} \tc 1 )
\end{split}
\]
and so we conclude that $D$ is a continuous $\FY$-derivation.

By Corollary  \ref{corfrobfinitplan} $F_{\FX/\FY\,*} \om^{\bullet}_{\FX/\FY}$ is a complex of locally free $\CO_{\FX^{(p)}}$-Modules  of finite rank and, in particular, $\CH^{i} (F_{\FX/\FY\,*}\om^{\bullet}_{\FX/\FY}) \in \A_{\cc}(\FX^{(p)})$, for all $i$. Applying \cite[Theorem 3.5]{AJP1}  it results that there exists an  unique homomorphism of $\CO_{\FX^{(p)}}$-Modules 
\[
\gamma^{1}\colon  \om^{1}_{\FX^{(p)}/\FY} \lto \CH^{1} (F_{\FX/\FY\,*}\om^{\bullet}_{\FX/\FY})
\]
 such that the following diagram is commutative:
\begin{equation*}
\begin{tikzpicture}[baseline=(current  bounding  box.center)]
\matrix(m)[matrix of math nodes, row sep=3em, column sep=3em,
text height=1.5ex, text depth=0.25ex]{
\CO_{\FX^{(p)}}                           & \om^{1}_{\FX^{(p)}/\FY} \\
\CH^{1} (F_{\FX/\FY\,*}\om^{\bullet}_{\FX/\FY}) &                    \\}; 
\path[->,font=\scriptsize,>=angle 90]
(m-1-1) edge node[auto] {$\hd'$} (m-1-2)
        edge node[left] {$D$} (m-2-1)
(m-1-2) edge node[auto] {$\gamma^{1}$} (m-2-1);
\end{tikzpicture}
\end{equation*}
Therefore,  applying again \cite[\emph{loc. cit.}]{boual} there exists an unique morphism of graded $\CO_{\FX^{(p)}}$-Algebras
\begin{equation*}
\gamma \colon
\bigoplus_{i \in \ZZ} \om^{i}_{\FX^{(p)}/\FY} \lto 
\bigoplus_{i \in \ZZ} \CH^{i} (F_{\FX/\FY\,*}\om^{\bullet}_{\FX/\FY})
\end{equation*}
that in degrees $0$ and  $1$ is defined by $\gamma^{0}$ and $\gamma^{1}$, respectively.
\end{cosa}

\begin{thm}  \label{teorisocar}
With hypothesis as in \ref{ImdirfrobcomRhamcoh}, the morphism $\gamma$ depicted in (\ref{ecisocart}) is an isomorphism and it is called the \emph{Cartier isomorphism}.
\end{thm}

\begin{proof}
We will do it in three steps:

(1) If $f \colon  X \to Y$ is a smooth morphism in $\sch$ with $Y$ of characteristic $p$, $\gamma$ is the Cartier isomorphism in $\sch$ (\cite[(7.2)]{k}).

(2) Let us  prove the result for  the canonical projection   $\pi_{\FY}\colon  \AF_{\FY}^{n} \to \FY$. Considering the diagram (\ref{diagrfrob})  for the morphisms $\pi_{\FY}$ and $\pi\colon   \AF_{\mathbb{F}_{p}}^{n}  \to \spec(\mathbb{F}_{p})$ and, keeping in mind \ref{frobespacafin}.(\ref{frobespacafiniso}) we have  the following commutative diagram in $\sfn$:
\begin{equation*}
\begin{tikzpicture}[baseline=(current  bounding  box.center), 
                    cross line/.style={preaction={-, draw=white, line width=6pt}}]
\matrix(m)[matrix of math nodes, row sep=2.8em, column sep=1.5em,
text height=1.5ex, text depth=0.25ex]{
\AF^{n}_{\FY} &      & (\AF^{n}_{\FY})^{(p)} &       & \AF^{n}_{\FY}       \\
& \AF_{\mathbb{F}_{p}}^{n} && (\AF_{\mathbb{F}_{p}}^{n})^{(p)}&& \AF_{\mathbb{F}_{p}}^{n} \\
\FY    &    & \FY              &      & \FY           \\
& \spec(\mathbb{F}_p) && \spec(\mathbb{F}_p)   && \spec(\mathbb{F}_p)\\}; 
\path[->,font=\scriptsize,>=angle 90]
(m-1-1) edge node[auto] {$F_{\AF^{n}_{\FY}/\FY}$} (m-1-3)
        edge node[above] {$g$} (m-2-2)
        edge node[near end, left] {$\pi_{\FY}$} (m-3-1)
(m-1-3) edge node[auto] {$(F_{\FY})_{\AF^{n}_{\FY}}$} (m-1-5)
        edge node[above] {$g'$} (m-2-4)
        edge node[near end, left] {$\pi_{\FY}^{(p)}$} (m-3-3)
(m-1-5) edge node[auto] {$g$} (m-2-6)
        edge node[near end, right] {$\pi_{\FY}$} (m-3-5)
(m-2-2) edge [cross line] node[near start, above] {$F_{\AF_{\mathbb{F}_{p}}^{n}/\mathbb{F}_p}$} (m-2-4)
(m-2-4) edge [cross line] node[near start, above] {$(F_{\mathbb{F}_p})_{\AF_{\mathbb{F}_{p}}^{n}}\quad$} (m-2-6)
(m-2-6) edge node[near end, right] {$\pi$} (m-4-6)
(m-3-1) edge node[near start, auto] {$\id_{\FY}$} (m-3-3)
        edge node[left] {$h\,$} (m-4-2)
(m-2-2) edge [cross line] node[near end, left] {$\pi$} (m-4-2)
(m-3-3) edge node[near start, auto] {$F_{\FY}$} (m-3-5)
        edge node[left] {$h$} (m-4-4)
(m-2-4) edge [cross line] node[near end, left] {$\pi^{(p)}$} (m-4-4)
(m-3-5) edge node[right] {$\,h$} (m-4-6)
(m-4-2) edge node[below] {$\id_{\mathbb{F}_p}$} (m-4-4)
(m-4-4) edge node[below] {$F_{\mathbb{F}_p}$} (m-4-6);
   \node at ([shift=({0em,2.3em})]m-2-4) {$\square_{1}$};
   \node at ([shift=({0em,2.3em})]m-2-2) {$\Diamond_{4}$};
   \node at ([shift=({1.8em,-2.3em})]m-2-4) {$\square_{2}$};
   \node at ([shift=({1.5em,-2.7em})]m-1-1) {$\Diamond_{3}$};
   \node at ([shift=({1.5em,-2.7em})]m-1-3) {$\Diamond_{2}$};
      \node at ([shift=({1.5em,-2.7em})]m-1-5) {$\Diamond_{1}$};
\end{tikzpicture}
\end{equation*} 
The squares $\square_{1},\, \square_{2}$ and $\Diamond_{1}$ are cartesian, therefore the square $\Diamond_{2}$  is also cartesian. Since $\Diamond_{3}$ is a cartesian square it results  that $\Diamond_{4}$ is also  cartesian.
Applying (1), we have the Cartier isomorphism asociated to the scheme morphism $\pi$:
\begin{equation} \label{isocarafinusual}
\gamma_{n,\mathbb{F}_{p}} \colon
\bigoplus_{i \in \ZZ} \Omega^{i}_{(\AF_{\mathbb{F}_{p}}^{n})^{(p)}/\mathbb{F}_{p}}
\liso
\bigoplus_{i \in \ZZ} \CH^{i} (F_{\AF_{\mathbb{F}_{p}}^{n}/\mathbb{F}_p\,*}\Omega^{\bullet}_{\AF_{\mathbb{F}_{p}}^{n}/\mathbb{F}_{p}})
\end{equation} 
Proposition \cite[Proposition 3.7]{AJP1} implies that $\om^{1}_{(\AF_{\FY}^{n})^{(p)}/\FY} \cong g'^{*} \Omega^{1}_{(\AF_{\mathbb{F}_{p}}^{n})^{(p)}/\mathbb{F}_{p}}  $ and, from the fact that $g'$ is a flat  morphism (by  base-change) and from the isomorphism (\ref{isocarafinusual}) we deduce the isomorphism
\begin{equation*} 
\gamma_{n,\FY} \colon 
\bigoplus_{i \in \ZZ} \om^{i}_{(\AF_{\FY}^{n})^{(p)}/\FY}
\liso  
\bigoplus_{i \in \ZZ} \CH^{i} (g'^{*} F_{\AF_{\mathbb{F}_{p}}^{n}/\mathbb{F}_p\,*}\Omega^{\bullet}_{\AF_{\mathbb{F}_{p}}^{n}/\mathbb{F}_{p}})
\end{equation*} 
By  \ref{frobespacafin}.(\ref{frobrelproyecfinitflat}) we have that $F$ is a finite  morphism  and then Proposition \ref{cambiobasefinito} applies. Therefore $ g'^{*} F_{\AF_{\mathbb{F}_{p}}^{n}/\mathbb{F}_p\,*}\Omega^{\bullet}_{\AF_{\mathbb{F}_{p}}^{n}/\mathbb{F}_{p}} \cong F_{\AF^{n}_{\FY}/\FY\,*}g^{*} \Omega^{\bullet}_{\AF_{\mathbb{F}_{p}}^{n}/\mathbb{F}_{p}}$ and we obtain the Cartier  isomorphism asociated to $\pi_{\FY}$
\begin{equation*} 
\gamma_{n,\FY} \colon
\bigoplus_{i \in \ZZ} \om^{i}_{(\AF_{\FY}^{n})^{(p)}/\FY}
\liso
\bigoplus_{i \in \ZZ} \CH^{i} (F_{\AF^{n}_{\FY}/\FY\,*}\om^{\bullet}_{\AF_{\FY}^{n}/\FY})
\end{equation*}

(3) In the general  case, since it is a local question by \cite[Proposition 5.9]{AJP2}
we may suppose  that  $f$ factors in $\pi \circ g\colon  \FX \to \AF^{n}_{\FY} \to \FY$,
where $g$ is \'etale and $\pi$ is the canonical projection. Considering the diagram (\ref{diagrfrob})  for  the morphisms $g$, $\pi$ and $f$ we have a commutative diagram of locally noetherian formal schemes (\ref{prisma}) where $\Diamond_{1},\, \square_{1},\, \Diamond_{2}$ and $\Diamond_{3}$ are cartesian squares. Notice that we use that $g$ is \'etale.

By (2), associated to the  morphism $\pi_{\FY}\colon  \AF^{n}_{\FY} \to \FY$, we have the Cartier  isomorphism :
\begin{equation} \label{isoCartierespacafin}
\gamma_{n,\FY} \colon
\bigoplus_{i \in \ZZ} \om^{i}_{(\AF^{n}_{\FY})^{(p)}/\FY}
\liso 
\bigoplus_{i \in \ZZ} \CH^{i} (F_{\AF^{n}_{\FY}/\FY\,*}\om^{\bullet}_{\AF^{n}_{\FY}/\FY})
\end{equation} 
Since $g$ is \'etale and $\Diamond_{3}$ is a cartesian square, by  base-change (\cite[Proposition 2.9]{AJP1}) we have that $g'$ is \'etale and from \cite[Corollary 4.10]{AJP1} we deduce that $g'^{*} \om^{i}_{(\AF^{n}_{\FY})^{(p)}/\FY} \cong \om^{i}_{\FX^{(p)}/\FY}$ for all $i \in \ZZ$.  Since $g'$ is flat and,  applying $g'^{*}$ to the isomorphism (\ref{isoCartierespacafin}) we have the following isomorphism
 \begin{equation*} 
\bigoplus_{i \in \ZZ} \om^{i}_{\FX^{(p)}/\FY}
\liso
\bigoplus_{i \in \ZZ}\CH^{i}( g'^{*}F_{\AF^{n}_{\FY}/\FY\,*}\om^{\bullet}_{\AF^{n}_{\FY}/\FY}).
\end{equation*}
On the other hand,   $g$ is \'etale and, from \cite[\emph{loc. cit}.]{AJP1} we deduce that $g^{*} \om^{i}_{\AF^{n}_{\FY}/\FY} \cong \om^{i}_{\FX/\FY}$ for all $i \in \ZZ$. Last, applying Proposition \ref{cambiobasefinito} it results that for all $i$
\[
F_{\FX/\FY\,*}  \om^{i}_{\FX/\FY} = F_{\FX/\FY\,*} g^{*} \om^{i}_{\AF^{n}_{\FY}/\FY} \cong g'^{*}F_{\AF^{n}_{\FY}/\FY\,*} \om^{i}_{\AF^{n}_{\FY}/\FY}
\] 
therefore $\CH^{i} (F_{\FX/\FY\,*}\om^{\bullet}_{\FX/\FY}) \cong \CH^{i} (g'^{*}F_{\AF^{n}_{\FY}/\FY\,*}\om^{\bullet}_{\AF^{n}_{\FY}/\FY})$ as wanted.
\end{proof}


\section{Decomposition Theorem up to $p$} \label{sec4}



\begin{cosa}
Recall that a complex $\CE \in \D(\FX)$  is \emph{decomposable} if it is isomorphic to a  complex in $\D(\FX)$ with zero differential. A \emph{decomposition} of $\CE$ is an isomorphism
\begin{equation} 
\CE  \cong \bigoplus_{i \in \ZZ} \CH^{i} \CE[-i] \textrm{ in } \D(\FX)
\end{equation}
that induces the identity between the homologies. 
\end{cosa}


\begin{cosa}
(\cfr{}\cite[3.1]{P})
 Let $\FW$ be a  locally noetherian formal scheme of characteristic $p$, $i \colon \FW \inc \FZ$ a  closed immersion given by a square zero ideal $\CI \subset \CO_{\FZ}$ and  $g \colon \FY \to \FW$ a flat (smooth) morphism. If there exists   a  flat (smooth) morphism $\tilde{g} \colon \widetilde{\FY} \to \FZ$ in $\sfn$ such that  the   diagram
\begin{equation*}
\begin{tikzpicture}[baseline=(current  bounding  box.center)]
\matrix(m)[matrix of math nodes, row sep=3em, column sep=3em,
text height=1.5ex, text depth=0.25ex]{
\widetilde{\FY}     & \FZ \\
\FY & \FW \\}; 
\path[->,font=\scriptsize,>=angle 90]
(m-1-1) edge[dashed] node[auto] {$\tilde{g}$} (m-1-2)
(m-2-1) edge node[auto] {$g$} (m-2-2);
\path[left hook->,font=\scriptsize,>=angle 90]
(m-2-1) edge (m-1-1)
(m-2-2) edge node[auto] {$i$}(m-1-2);
\end{tikzpicture}
\end{equation*}
is cartesian we will say that $\widetilde{\FY}$, or that the  morphism $\FY \inc \widetilde{\FY}$, or that $\tilde{g}$ is a 
\emph{flat (smooth) lifting of $\FY$ over $\FZ$}. 

Whenever $\FW = \spec(\mathbb{F}_{p})$ and $\FZ=\spec(\ZZ/p^{2} \ZZ)$ we will say that $\widetilde{\FY}$ is  flat (smooth) lifting of $\FY$ over $\ZZ/p^{2} \ZZ$.
\end{cosa}
 
The following is one of the main results of this paper. It extends to formal schemes the classical Decomposition Theorem in \cite[Corollaire 3.7.(a)]{deill} (see also \cite[Th\'eor\`eme 5.1]{ill}).

\begin{thm}[Decomposition Theorem] \label{teoremadedescomposicion}
Let $\FY$ be a  locally noetherian formal scheme of characteristic $p$ and $\widetilde{\FY}$ a flat lifting of $\FY$ over $\ZZ/p^{2} \ZZ$. Let $f \colon  \FX \to \FY$ be a smooth morphism of locally noetherian formal schemes. Any smooth lifting $\widetilde{\FX^{(p)}}$ of $\FX^{(p)}$ over $\widetilde{\FY}$ provides a decomposition of the complex $\tau^{<p}(F_{\FX/\FY\,*} \om^{\bullet}_{\FX/\FY})$ in $\D(\FX^{(p)})$, where $F_{\FX/\FY} \colon  \FX \to \FX^{(p)}$ denotes the relative Frobenius morphism of $\FX$ over $\FY$. 
\end{thm}

\begin{rem}
Mimicking the proof of \cite[Th\'eor\`eme 3.5]{deill} we can show a converse to the theorem, specifically, a decomposition of $\tau^{<p}(F_{\FX/\FY\,*} \om^{\bullet}_{\FX/\FY})$ provides a smooth lifting $\widetilde{\FX^{(p)}}$ of $\FX^{(p)}$ over $\widetilde{\FY}$. We leave the details to the interested reader
\end{rem}

We defer the proof of Theorem \ref{teoremadedescomposicion} to the next section. In the next few paragraphs we will present some consequences. We will start establishing some notations.

\begin{cosa}
Let $k$  be a perfect field of characteristic $p$ and put $Y= \spec(k)$. Then there exists a flat lifting  of $Y$ over $\ZZ/p^{2} \ZZ$ given (up to isomorphism) by  $\widetilde{Y} = \spec(W_{2}(k))$ where $W_{2}(k)$ is the ring of Witt vectors of length $2$ over  $k$. 
On the other hand, the absolute Frobenius  endomorphism  of $F_k = F_{Y}\colon  Y \to Y$ is an automorphism. So, given $f\colon  \FX \to Y$ a smooth morphism in $\sfn$ from the corresponding diagram (\ref{diagrfrob}) we deduce that $(F_{k})_{\FX}\colon  \FX^{(p)} \to \FX$ is an isomorphism. Then $\FX^{(p)}$ admits a smooth lifting over $\widetilde{Y}$ if, and only if, $\FX$ also  does.
\end{cosa}

\begin{cor} \label{cor1descomposicion}
Given $k$ a perfect field of characteristic $p$, let $f \colon  \FX\to Y= \spec(k)$ be a smooth  morphism  in $\sfn$. If there  exists a smooth lifting of $\FX$ over $\widetilde{Y}= \spec(W_{2}(k))$, then $\tau^{<p}(F_{\FX/k\,*}\om^{\bullet}_{\FX/k})$ is decomposable in $\D(\FX^{(p)})$.
\end{cor}

\begin{rem}
 This corollary generalizes \cite[Th\'eor\`eme 2.1]{deill} to the context of formal schemes.
\end{rem}

\begin{cor} \label{cor2descomposicion}
Given a perfect field $k$ of characteristic $p$, let $f \colon  Z \to Y= \spec(k)$ be a morphism of finite type in $\sch$ and suppose that $Z$ is embeddable in a smooth $Y$-scheme  $X$. If  there exists a smooth lifting of $\widehat{X}:=X_{/Z}$ over $\widetilde{Y}= \spec(W_{2}(k))$, then $\tau^{<p}(F_{\widehat{X}/k\,*} \om^{\bullet}_{\widehat{X}/k})$ is decomposable in $\D(\widehat{X}^{(p)})$. 
\end{cor}

\begin{cor} \label{cor3descomposicion}
Given a perfect field $k$ of characteristic $p$, let $Z$ be a projective $k$-scheme  embeddable in $\PR:=\PR^{n}_{k}$ and let $\HPR :=\PR_{/Z}$. Then $\tau^{<p}(F_{\HPR/k\,*}  \om^{\bullet}_{\HPR/k})$ is decomposable in $\D(\HPR^{(p)})$. 
\end{cor}

\begin{proof}
Since $\PR^{n}_{W_{2}(k)}=\PR \times_{k}\spec(W_{2}(k))$, $Z$ is also a  closed subscheme of $\PR^{n}_{W_{2}(k)}$. If $\kk \colon \HPR^{n}_{W_{2}(k)} \to \PR^{n}_{W_{2}(k)}$ is the completion morphism of $\PR^{n}_{W_{2}(k)}$ along $Z$, then by  \cite[Proposition 3.10]{AJP2} it is immediate that the composition
\[
 \HPR^{n}_{W_{2}(k)} \overset{\kk}\lto \PR^{n}_{W_{2}(k)} \lto \spec(W_{2}(k))
\]
is a smooth lifting of $\HPR$ over $\spec(W_{2}(k))$.
\end{proof}

\section{Proof of the Decomposition Theorem}

The proof of Theorem \ref{teoremadedescomposicion} will be decomposed into several intermediate steps. We will mostly follow the strategy of the proof of the Decomposition Theorem for usual schemes in \cite[\S5]{ill}.

\begin{cosa} \label{intrdemosdesc}
Recall that a decomposition of $\tau^{<p}(F_{\FX/\FY\,*} \om^{\bullet}_{\FX/\FY})$ is equivalent to give a morphism in $\D(\FX^{(p)})$
\begin{equation*}
\bigoplus_{i<p} \CH^{i} (F_{\FX/\FY\,*}  \om^{\bullet}_{\FX/\FY} [-i])  \lto  F_{\FX/\FY\,*}  \om^{\bullet}_{\FX/\FY}
\end{equation*} 
that induces the identity through the functor $\CH^{i}$ for all $i<p$. By Theorem  \ref{teorisocar} it is sufficient to give a morphism in $\D(\FX^{(p)})$ 
\begin{equation} \label{isodescomcart}
\bigoplus_{i<p}  \om^{i}_{\FX^{(p)}/\FY} [-i]  \lto  F_{\FX/\FY\,*}  \om^{\bullet}_{\FX/\FY}
\end{equation}
that  coincides in  homology with  the Cartier isomorphism. 

We will associate a morphism as (\ref{isodescomcart}) to each smooth  lifting $\widetilde{\FX^{(p)}}$ of $\FX^{(p)}$ over $\widetilde{\FY}$.
The proof proceeds in two stages:
\begin{enumerate}
 \item First we show (in Proposition \ref{teoremdescompslevanglobal}) that if there exists a 
 \emph{global} lifting of Frobenius, \ie\/a  $\widetilde{\FY}$-morphism 
 \[
 \widetilde{F} \colon  \widetilde{\FX} \to \widetilde{\FX^{(p)}}
 \]
 that lifts $F_{\FX/\FY}$  (see \ref{defFlevantaaF0}), then the complex $\tau^{<p}(F_{\FX/\FY\,*} \om^{\bullet}_{\FX/\FY})$ is decomposable in $\D(\FX^{(p)})$ by constructing a lifting of the Cartier operator, see \ref{deffisuper1}.
 \item Liftings of Frobenius only exist locally, this is discussed in \ref{localglobal}. With this, we see (in Proposition \ref{prodescomtruncadodos}) that $\tau^{\leq 1}(F_{\FX/\FY\,*} \om^{\bullet}_{\FX/\FY})$ is decomposable in $\D(\FX^{(p)})$ by pasting these local liftings.
Finally, we extend this decomposition to the whole $\tau^{<p}(F_{\FX/\FY\,*} \om^{\bullet}_{\FX/\FY})$ using the multiplicative structure of the De Rham complex (Proposition \ref{propdescomtruncadop}).
\end{enumerate}
\end{cosa}

We start by fixing some notations and definitions.

\begin{cosa} \label{cosaisooyopoy} \textbf{Two canonical isomorphisms}.
Let  $i \colon \FX  \inc \widetilde{\FX}$ be a smooth lifting  over $\widetilde{\FY}$. 
From the short exact sequence of $(\ZZ/p^{2} \ZZ)$-modules 
\[
0 \to p \cdot \ZZ/p^{2} \ZZ\to \ZZ/p^{2} \ZZ \to \mathbb{F}_{p} \to 0
\]
we deduce  that the sequence  of $\CO_{\widetilde{\FX}}$-modules
\begin{equation}\label{s1}
0 \lto p \cdot \CO_{\widetilde{\FX}}\lto \CO_{\widetilde{\FX}} \lto i_{*}(\CO_\FX) \lto 0
\end{equation}
is exact and  therefore  $i$ is a closed embedding given  by  the ideal $p \cdot \CO_{\widetilde{\FX}} \subset \CO_{\widetilde{\FX}}$.

The isomorphism $p \cdot \colon \mathbb{F}_{p} \to p \cdot \ZZ/p^{2} \ZZ$ of $(\ZZ/p^{2} \ZZ)$-modules induces the isomorphism of $\CO_{\widetilde{\FX}}$-Modules
\begin{equation} \label{iso1}
\begin{array}{ccc}
\mathbf{p}^0 \colon i_{*}(\CO_{\FX})& \lto &p\cdot \CO_{\widetilde{\FX}}\\
\end{array}
\end{equation}
locally determined by  $a + p \cdot \CO_{\widetilde{\FX}}  \leadsto  p \cdot a$.
Since $\om^{1}_{\widetilde{\FX}/\widetilde{\FY}}$ is a locally free $\CO_{\widetilde{\FX}}$-Module (see \cite[Proposition 2.6.1]{LNS}) applying the functor $- \otimes_{\CO_{\widetilde{\FX}}} \om^{1}_{\widetilde{\FX}/\widetilde{\FY}}$ to the  sequence (\ref{s1})  and to the isomorphism (\ref{iso1}) we obtain  the short exact sequence of $\CO_{\widetilde{\FX}}$-Modules
\begin{equation}\label{s2}
0 \lto p \cdot \om^{1}_{\widetilde{\FX}/\widetilde{\FY}} 
\lto \om^{1}_{\widetilde{\FX}/\widetilde{\FY}} 
\lto i_{*}(\CO_{\FX}) \otimes_{\CO_{\widetilde{\FX}}} \om^{1}_{\widetilde{\FX}/\widetilde{\FY}}  \lto 0
\end{equation}
and the isomorphism of  $\CO_{\widetilde{\FX}}$-Modules
\begin{equation} \label{iso2}
\begin{array}{cccc}
\mathbf{p}^1 \colon  i_{*}(\CO_{\FX}) \otimes_{\CO_{\widetilde{\FX}}} \om^{1}_{\widetilde{\FX}/\widetilde{\FY}}& \lto &p\cdot \om^{1}_{\widetilde{\FX}/\widetilde{\FY}}\\
\end{array}
\end{equation}
Observe that the isomorphism $\mathbf{p}^1$ is locally  defined by $1 \otimes \hd(s)  \leadsto  p \cdot \hd(s)$.
\end{cosa}

\begin{cosa} \label{defFlevantaaF0} \textbf{Liftings of Frobenius}.
From now on we will assume the set-up and hypotheses of Theorem \ref{teoremadedescomposicion}.
Given $F_{\FX/\FY}\colon  \FX \to \FX^{(p)}$ the relative Frobenius morphism of $\FX$ over $\FY$ let us  suppose  that there exist $i \colon \FX \inc \widetilde{\FX}$ and $i' \colon \FX^{(p)} \inc \widetilde{\FX^{(p)}}$ smooth liftings over $\widetilde{\FY}$. We say that  a $\widetilde{\FY}$-morphism $\widetilde{F} \colon  \widetilde{\FX} \to \widetilde{\FX^{(p)}}$  \emph{is a lifting\footnote{According to the terminology established in \cite[\S 2]{P} we would  say that $\widetilde{F}$ is a lifting of $\FX \xto{F_{\FX/\FY}} \FX^{(p)} \overset{i'}\inc \widetilde{\FX^{(p)}}$ over $\widetilde{\FY}$.} of $F_{\FX/\FY}$} if the following diagram is commutative
\begin{equation}\label{ccarpropdescom}
\begin{tikzpicture}[baseline=(current  bounding  box.center)]
\matrix(m)[matrix of math nodes, row sep=3em, column sep=3em,
text height=1.5ex, text depth=0.25ex]{
\widetilde{\FX}     & \widetilde{\FX^{(p)}} \\
\FX & \FX^{(p)} \\}; 
\path[->,font=\scriptsize,>=angle 90]
(m-1-1) edge node[auto] {$\widetilde{F}$} (m-1-2)
(m-2-1) edge node[auto] {$F_{\FX/\FY}$} (m-2-2);
\path[left hook->,font=\scriptsize,>=angle 90]
(m-2-1) edge node[auto] {$i$} (m-1-1)
(m-2-2) edge node[right] {$i'$} (m-1-2);
\end{tikzpicture}
\end{equation} 
Observe that, since $\FX \cong \widetilde{\FX} \times_{\widetilde{\FY}}  \FY$ and $\FX^{(p)} \cong \widetilde{\FX^{(p)}}\times_{\widetilde{\FY}}  \FY$ we have  that the square (\ref{ccarpropdescom}) is cartesian. 
\end{cosa}

\begin{lem}\label{restrilevanapnula}
The image of the canonical morphism   
\[
\om^{1}_{\widetilde{\FX^{(p)}}/\widetilde{\FY}} \lto \widetilde{F}_{*} \om^{1}_{\widetilde{\FX}/\widetilde{\FY}}
\]
 is contained in $p \cdot (\widetilde{F}_{*} \om^{1}_{\widetilde{\FX}/\widetilde{\FY}})$.
\end{lem}

\begin{proof}

Indeed, the morphism of $\CO_{\widetilde{\FX^{(p)}}}$-Modules 
\[
 i'_{*}\CO_{\FX^{(p)}}\otimes_{\CO_{\widetilde{\FX^{(p)}}}} \om^{1}_{\widetilde{\FX^{(p)}}/\widetilde{\FY}}    \lto 
 i'_{*}\CO_{\FX^{(p)}}\otimes_{\CO_{\widetilde{\FX^{(p)}}}} \widetilde{F}_{*} \om^{1}_{\widetilde{\FX}/\widetilde{\FY}}
 \]
 corresponds through the projection formula \cite[\textbf{0}, (5.4.8)]{EGA1} to
 \[
i'_{*}\om^{1}_{\FX^{(p)}/\FY}  \lto 
i'_{*}F_{\FX/\FY\,*} \om^{1}_{\FX/\FY} \cong
\widetilde{F}_{*}i_{*}\om^{1}_{\FX/\FY},
\]
and this map is zero by \ref{relativonulo}. We conclude since $i'_{*}\CO_{\FX^{(p)}} = \CO_{\widetilde{\FX^{(p)}}}/p\cdot \CO_{\widetilde{\FX^{(p)}}}$.
\end{proof}

\begin{cosa} \label{deffisuper1} \textbf{A Cartier operator.}
Under the hypotheses and notations of \ref{defFlevantaaF0}, applying $i'^{*}$ to the canonical morphism   $\om^{1}_{\widetilde{\FX^{(p)}}/\widetilde{\FY}} \to \widetilde{F}_{*} \om^{1}_{\widetilde{\FX}/\widetilde{\FY}}$ we have that there  exists an unique morphism of $\CO_{\FX^{(p)}}$-Modules  
\begin{equation}\label{cartphi}
 \varphi_{\widetilde{F}}^{\, 1}\colon  \om^{1}_{\FX^{(p)}/\FY}  \lto 
 F_{\FX/\FY\,*} \om^{1}_{\FX/\FY}
\end{equation}
 such that the following diagram is commutative
\begin{equation}\label{diagradeffi1}
\begin{tikzpicture}[baseline=(current  bounding  box.center)]
\matrix(m)[matrix of math nodes, row sep=3em, column sep=3em,
text height=2ex, text depth=1.5ex]{
i'^{*}\om^{1}_{\widetilde{\FX^{(p)}}/\widetilde{\FY}} & p \cdot i'^{*} \widetilde{F}_{*} \om^{1}_{\widetilde{\FX}/\widetilde{\FY}} \\
\om^{1}_{\FX^{(p)}/\FY} & F_{\FX/\FY\,*} \om^{1}_{\FX/\FY}         \\};
\path[->,font=\scriptsize,>=angle 90]
(m-1-1) edge node[auto] {can.} (m-1-2)
(m-1-2) edge (m-2-2)
(m-2-1) edge node[auto] {$\varphi_{\widetilde{F}}^{\, 1}$} (m-2-2)
        edge node[left] {$\wr$} (m-1-1.south);
\end{tikzpicture}
\end{equation}
where the left vertical isomorphism is given by  base-change (see  \cite[Proposition 3.7]{AJP1}), and the right vertical morphism corresponds to the isomorphism  
\[
 p \cdot \widetilde{F}_{*}\om^{1}_{\widetilde{\FX}/\widetilde{\FY}}
 \xto{(\widetilde{F}_{*}\mathbf{p}^1)^{-1}} 
 \widetilde{F}_{*}(i'_{*}\CO_{\FX} \otimes_{\CO_{\widetilde{\FX}}} \om^{1}_{\widetilde{\FX}/\widetilde{\FY}}) 
 \cong(\widetilde{F} \circ i)_{*} \om^{1}_{\FX/\FY}
 \] 
through the adjoint pair $i'^{*} \dashv i'_{*}$. Let us call $\varphi_{\widetilde{F}}^{\, 1}$ the Cartier operator.

Let us  give a local description of the morphism $\varphi_{\widetilde{F}}^{\, 1}$. For that,
assume that we are in $\sfna$ and set $\FY=\spf(B)$, $\widetilde{\FY}=\spf(\widetilde{B})$, $\FX=\spf(A)$, $\widetilde{\FX}=\spf(\widetilde{A})$,  $\FX^{(p)}=\spf(A^{(p)})$ and  $\widetilde{\FX^{(p)}}=\spf(\widetilde{A^{(p)}})$ with $A=\widetilde{A}/p\widetilde{A}$ and $A^{(p)}=\widetilde{A^{(p)}}/p\widetilde{A^{(p)}}$. Now, given  $a=a_1 + p \cdot \widetilde{A}$ with $a_1 \in \widetilde{A}$ and  $a_2 \in \widetilde{A^{(p)}}$ such that $a \otimes 1=a_2+ p \cdot \widetilde{A^{(p)}} $, since $F(a \otimes 1)=  a^{p}$ (see \ref{frobrelafin}) from the  conmutativity of  diagram (\ref{ccarpropdescom}) we deduce that
\[
F(a_2)= a_1^{p}+p\cdot c_1
\] 
with $c_1\in \widetilde{A}$.
From this we deduce that $\varphi_{\widetilde{F}}^{\, 1}$ is locally given by  
\[
\hd(a) \otimes 1   \leadsto a^{p-1} \hd(a) + \hd(c)
\]
 where $c=c_1+p\cdot \widetilde{A}$.
\end{cosa}

\begin{lem}\label{varfiocar}
In the setting of  \ref{defFlevantaaF0}, the Cartier operator $\varphi_{\widetilde{F}}^{\, 1}$ 
defined in (\ref{cartphi}) induces in homology the Cartier isomorphism in degree $1$. 
\end{lem}
\begin{proof}
From the local description of $\varphi_{\widetilde{F}}^{\, 1}$ just given, 
we  deduce that $\Img (\varphi_{\widetilde{F}}^{\, 1}) \subset \CZ^{1} F_{\FX/\FY\,*}\om^{\, \bullet}_{\FX/\FY}$ and that the composition of morphisms  
\[
\om^{1}_{\FX^{(p)}/\FY} 
\xto{\varphi_{\widetilde{F}}^{\, 1}}  \CZ^{1} (F_{\FX/\FY\,*}\om^{\, \bullet}_{\FX/\FY})
\xepi{\quad} \CH^{1} (F_{\FX/\FY\,*}\om^{\, \bullet}_{\FX/\FY})
\]  
is $\gamma^{1}$, the Cartier isomorphism  (\ref{ecisocart}) in degree $1$.
\end{proof}

\begin{prop} \label{teoremdescompslevanglobal}
Suppose that there exists a $\widetilde{\FY}$-morphism $\widetilde{F}$ that lifts  $F_{\FX/\FY}$.
Then there exist a morphism in the category of complexes of objects in $\CA(\FX^{(p)})$ 
\[
\varphi_{\widetilde{F}} \colon \bigoplus_{i<p}  \om^{i}_{\FX^{(p)}/\FY} [-i]  \lto  F_{\FX/\FY\,*}  \om^{\bullet}_{\FX/\FY}
\]
that induces the Cartier isomorphism  (\ref{ecisocart}) in $\CH^{i}$, for all $i<p$, such that $\varphi^{0}_{\widetilde{F}} = F_{\FX/\FY}^\sharp$ and the morphism $\varphi^{1}_{\widetilde{F}}$ is the one defined in \ref{deffisuper1}.
\end{prop}

\begin{proof}
By Lema \ref{varfiocar} and Theorem \ref{teorisocar}, it suffices to take  $\varphi^{\, i}_{\widetilde{F}}$ the composition of the morphisms 
\[\omi_{\FX^{(p)}/\FY} = \wedge^{i} \om^{1}_{\FX^{(p)}/\FY} \xto{\wedge^{i} \varphi^{\, 1}_{\widetilde{F}}} \wedge^{i} F_{\FX/\FY\,*} \om^{1}_{\FX/\FY} \xto{\textrm{prod.}} F_{\FX/\FY\,*} \om^{i}_{\FX/\FY} \textrm{,}
\] for all $1<i<p$.
\end{proof}

\begin{cor}
If there exists a $\widetilde{\FY}$-morphism $\widetilde{F}$ that lifts  $F_{\FX/\FY}$ then there is a decomposition of  $\tau^{<p}(F_{\FX/\FY\,*} \om^{\bullet}_{\FX/\FY})$ in $\D(\FX^{(p)})$.
\end{cor}

\begin{proof}
By \ref{intrdemosdesc} it is an immediate consequence of the proposition.
\end{proof}

\begin{cosa}\label{localglobal}
 Having dealt with the case in which there is a global lifting of Frobenius, we treat now the general case of Theorem \ref{teoremadedescomposicion}. We start by showing that the complex $\tau^{\leq 1}(F_{\FX/\FY\,*} \om^{\bullet}_{\FX/\FY})$ is decomposable in $\D(\FX^{(p)})$. For that, given an arbitrary affine open covering $\{\FU_{\alpha}\}$ of $\FX$, by  \cite[Corollary 4.3]{P} for  each $\alpha$ there exists a smooth lifting $\widetilde{\FU}_{\alpha}$ of $\FU_{\alpha}$ over $\widetilde{\FY}$. Furthermore, \cite[Corollary 2.5]{AJP1} implies that there exists a lifting $\widetilde{F}_{\alpha}\colon  \widetilde{\FU}_{\alpha} \to \widetilde{\FX^{(p)}}$ of $F_{\FX/\FY}|_{\FU_{\alpha}}\colon  \FU_{\alpha} \to \FX^{(p)} \inc \widetilde{\FX^{(p)}}$ over $\widetilde{\FY}$. We are going to  ``glue"  in $\D(\FX^{(p)})$ the morphisms  $\varphi_{\widetilde{F}_{\alpha}}$ asociated to  each  lifting $\widetilde{F}_{\alpha}$  (\emph{cf.}  Proposition \ref{teoremdescompslevanglobal}) and we will check that  does not depend of the chosen covering of $\FX$. This construction is not trivial due to the lack of the local nature of the derived category. 
\end{cosa}

We need the following lemma in which we compare the morphisms $\varphi_{\widetilde{F}}$ asociated  to different liftings $\widetilde{F}$ of $F_{\FX/\FY}$.

\begin{lem}\label{lemadatospegado}
Suppose given $\widetilde{F}_{1}\colon  \widetilde{\FX}_{1} \to \widetilde{\FX^{(p)}}$ and $\widetilde{F}_{2}\colon  \widetilde{\FX}_{2} \to \widetilde{\FX^{(p)}}$ a pair of $\widetilde{\FY}$-morphisms  that lift $F_{\FX/\FY}$, then there exists an homomorphism of $\CO_{\FX^{(p)}}$-Modules $\phi(\widetilde{F}_{1},\widetilde{F}_{2})\colon  \om^{1}_{\FX^{(p)} /\FY} \to F_{\FX/\FY\,*}\CO_{\FX}$ such that:
\begin{equation} \label{econdicpegadofi11}
\varphi_{\widetilde{F}_{1}}^{\, 1}-\varphi^{\, 1}_{\widetilde{F}_{2}}= F_{\FX/\FY\,*}\hd \circ \phi(\widetilde{F}_{1},\widetilde{F}_{2})
\end{equation}
Moreover given $\widetilde{F}_{3}\colon   \widetilde{\FX}_{3} \to \widetilde{\FX^{(p)}}$ another $\widetilde{\FY}$-morphism that lifts $F$, the cocycle condition holds, namely  
\begin{equation} \label{econdicpegadofi12}
\phi(\widetilde{F}_{1},\widetilde{F}_{2}) + \phi(\widetilde{F}_{2},\widetilde{F}_{3}) = \phi(\widetilde{F}_{1},\widetilde{F}_{3})
\end{equation}
\end{lem}

\begin{proof}
First, we are going to define $\phi(\widetilde{F}_{1},\widetilde{F}_{2})$ whenever there is a $\widetilde{\FY}$-iso\-mor\-phism $\tilde{u}\colon \widetilde{\FX}_{1}\to \widetilde{\FX}_{2}$ that induces the identity on $\FX$ (\emph{cf.}  \cite[3.4]{P}). The morphisms $\widetilde{F}_{1}$ and $\widetilde{F}_{2} \circ \tilde{u}$ are two liftings over $\widetilde{\FY}$ of the composed map 
\[
\FX \xto{F_{\FX/\FY}} \FX^{(p)} \overset{i'} \inc \widetilde{\FX^{(p)}},
\] 
and by \cite[2.2.(1)]{P} there exists an unique homomorphism of $\CO_{\widetilde{\FX^{(p)}}}$-Modules $\Psi\colon  \om^{1}_{\widetilde{\FX^{(p)}}/\widetilde{\FY}} \to \widetilde{F}_{1*}(p\cdot \CO_{\widetilde{\FX}})$ such that the diagram
\begin{equation*}
\begin{tikzpicture}[baseline=(current  bounding  box.center)]
\matrix(m)[matrix of math nodes, row sep=2.5em, column sep=3em,
text height=1.5ex, text depth=1ex]{
\CO_{\widetilde{\FX^{(p)}}}               & \om^{1}_{\widetilde{\FX^{(p)}}/\widetilde{\FY}} \\
\widetilde{F}_{1*}(p\cdot \CO_{\widetilde{\FX}}) &                    \\}; 
\path[->,font=\scriptsize,>=angle 90]
(m-1-1) edge node[auto] {$\hd'$} (m-1-2)
        edge node[left] {$\widetilde{F}_{1}^{\sharp} -(\widetilde{F}_{2} \circ \tilde{u})^{\sharp}$} (m-2-1);
\path[dashed, ->,font=\scriptsize,>=angle 90]
(m-1-2) edge node[auto] {$\Psi$} (m-2-1);
\end{tikzpicture}
\end{equation*}
is commutative. Applying $i'^{*}$ to the above  diagram we have that there exists a homomorphism of $\CO_{\FX^{(p)}}$-Modules $\phi(\tilde{u},\widetilde{F}_{1},\widetilde{F}_{2})\colon   \om^{1}_{\FX^{(p)}/Y}\to F_{\FX/\FY\,*} \CO_{\FX}$ such that the following diagram  commutes:
\begin{equation*}
\begin{tikzpicture}[baseline=(current  bounding  box.center)]
\matrix(m)[matrix of math nodes, row sep=4em, column sep=4.5em,
text height=1.5ex, text depth=0.25ex]{
i'^{*}\CO_{\widetilde{\FX^{(p)}}}     & i'^{*}\om^{1}_{\widetilde{\FX^{(p)}}/\widetilde{\FY}} & \om^{1}_{\FX^{(p)}/Y} \\
i'^{*}\widetilde{F}_{1 *} (p\cdot \CO_{\widetilde{\FX_{1}}}) & i'^{*}i'_{*}F_{\FX/\FY\,*}\CO_{\FX} 
     & F_{\FX/\FY\,*} \CO_{\FX}    \\}; 
\path[->,font=\scriptsize,>=angle 90]
(m-1-1) edge (m-1-2)
        edge node[left] {$\tau_1$} (m-2-1)
(m-1-2) edge node[above] {$i^{*}(\Psi)$} (m-2-1)
        edge node[auto] {$\sim$} (m-1-3)
(m-1-3) edge node[left] {$\phi(\tilde{u},\widetilde{F}_{1},\widetilde{F}_{2})$} (m-2-3)
(m-1-2) edge node[above] {$i^{*}(\Psi)$} (m-2-1)
(m-2-1) edge node[above] {$\tau_2$} (m-2-2)
(m-2-2) edge node[above] {nat.} (m-2-3);
\end{tikzpicture}
\end{equation*}
where $\tau_1:= i'^{*}(\widetilde{F}_{1}^{\sharp} -(\widetilde{F}_{2} \circ \tilde{u})^{\sharp})$ and $\tau_2:= i'^{*}\widetilde{F}_{1 *}((\mathbf{p}^0)^{-1})$.
Let us show that $\phi(\tilde{u},\widetilde{F}_{1},\widetilde{F}_{2})$ does not depend on $\tilde{u}$. Indeed, given $\tilde{v}\colon \widetilde{\FX}_{1}\to \widetilde{\FX}_{2}$ another $\widetilde{\FY}$-isomorphism that induces the identity on $\FX$, \cite[2.2.(1)]{P} implies that there exists an unique homomorphism of $\CO_{\widetilde{\FX}_{2}}$-Modules 
\[
\psi\colon  \om^{1}_{\widetilde{\FX}_{2}/\widetilde{\FY}} \lto \tilde{u}_{*}(p\cdot \CO_{\widetilde{\FX_{1}}}) \overset{\mathbf{p}^0} \cong i_{2*}\CO_{\FX}\]
such that  $\tilde{v}^{\sharp}-\tilde{u}^{\sharp}= \psi \circ \hd$ being $i_{2}\colon \FX \inc \widetilde{\FX}_{2}$ the inclusion. Equivalently by adjunction and, with an abuse of notation, there exists an unique homomorphism $\psi\colon  \om^{1}_{\FX/\FY} \to\CO_{\FX}$ of  $\CO_{\FX}$-Modules such that $\tilde{v}^{\sharp}-\tilde{u}^{\sharp}= \psi \circ \hd$. On the other hand, since $\widetilde{F}_{2} \circ \tilde{u}$ and $\widetilde{F}_{2} \circ \tilde{v}$ are two liftings of $i' \circ F_{\FX/\FY}$ over $\widetilde{\FY}$ by \cite[2.2.(1)]{P} there exists an unique morphism $\eta\colon  F_{\FX/\FY}^{*} \om^{1}_{\FX^{(p)}/\FY} \to  \CO_{\FX}$ of $\CO_{\FX}$-Modules such that $(\widetilde{F}_{2} \circ \tilde{v})^{\sharp}-(\widetilde{F}_{2} \circ \tilde{u})^{\sharp}= \eta \circ \hd'$. By unicity $\eta$ factors as
\[ 
F_{\FX/\FY}^{*}\om^{1}_{\FX^{(p)}/\FY} \xto{\text{can.}}  \om^{1}_{\FX/\FY} \xto{\,\psi\,} \CO_{\FX} 
\]
By \ref{relativonulo}  the canonical morphism $F_{\FX/\FY}^{*} \om^{1}_{\FX^{(p)}/\FY} \to \om^{1}_{\FX/\FY}$ is zero and we conclude  that  $\eta=0$ and, therefore, $\widetilde{F}_{2} \circ \tilde{u}= \widetilde{F}_{2} \circ \tilde{v}$. 

 In general, given an affine open covering $\{\FU_{\alpha}\}$ of $\FX$, for all $\alpha$, \cite[3.3]{P} implies that  there exists  a $\widetilde{\FY}$-isomorphism $\tilde{u}_{\alpha}\colon  \widetilde{\FX}_{1}|_{\FU_{\alpha}} \to \widetilde{\FX}_{2}|_{\FU_{\alpha}}$ that induces the identity on $\FU_{\alpha}$.  Then it suffices to  define for each $\alpha$ \[\phi(\widetilde{F}_{1},\widetilde{F}_{2})|_{\FU_{\alpha}} := \phi(\tilde{u}_{\alpha},\widetilde{F}_{1}|_{\FU_{\alpha}},\widetilde{F}_{2}|_{\FU_{\alpha}})\] 

To check the equalities (\ref{econdicpegadofi11}) and (\ref{econdicpegadofi12}) we may restrict to the affine case. In this case $\widetilde{\FX}_{1}$ and $\widetilde{\FX}_{2}$ are isomorphic (see \cite[3.3]{P}) and to simplify we set $\widetilde{\FX}:= \widetilde{\FX}_{1}=\widetilde{\FX}_{2}$. With notations as in \ref{deffisuper1}, we have that $\widetilde{F}_{i}(\widetilde{a^{(p)}}) = \tilde{a}^{p} + p\cdot c_{i}$ with $c_{i}\in \widetilde{A}$ for $i=1,2$, from where we deduce that 
\[\varphi_{\widetilde{F}_{1}}^{\, 1}-\varphi^{\, 1}_{\widetilde{F}_{2}}= F_{\FX/\FY\,*}\hd \circ \phi(\tilde{u},\widetilde{F}_{1},\widetilde{F}_{2})\]

Last, if we suppose there exists yet another $\widetilde{\FY}$-morphism $\widetilde{F}_{3}\colon \widetilde{\FX}_{3}\to \widetilde{\FX^{(p)}}$ that  lifts  $F$ and that $\tilde{v}\colon \widetilde{\FX}_{2}\to  \widetilde{\FX}_{3}$ is a $\widetilde{\FY}$-isomorphism that induces the identity in $\FX$, the equality  (\ref{econdicpegadofi12}) holds by adding the relations corresponding to the couples $(\widetilde{F}_{1},\widetilde{F}_{2})$ and $(\widetilde{F}_{2},\widetilde{F}_{3})$.
\end{proof}

\begin{prop} \label{prodescomtruncadodos}
There exists a morphism in $\D(\FX^{(p)})$ 
\[ \varphi^{\, 1} \colon \om^{1}_{\FX^{(p)}/\FY} [-1]  \lto  
F_{\FX/\FY\,*}  \om^{\bullet}_{\FX/\FY}
\]
that induces the Cartier isomorphism (\ref{ecisocart}) in $\CH^{1}$. 
\end{prop}

\begin{proof}
Let us fix an affine open covering $\{\FU_{\alpha}\}$ of $\FX$. By \ref{localglobal} there exists a smooth lifting $\widetilde{\FU}_{\alpha}$ of $\FU_{\alpha}$ over $\widetilde{\FY}$ and a lifting $\widetilde{F}_{\alpha}\colon  \widetilde{\FU}_{\alpha} \to \widetilde{\FX^{(p)}}$ of $F_{\FX/\FY}|_{\FU_{\alpha}}\colon  \FU_{\alpha} \to \FX^{(p)} \inc \widetilde{\FX^{(p)}}$ over $\widetilde{\FY}$, that is, such that the following diagram is commutative:
\begin{equation*}
\begin{tikzpicture}[baseline=(current  bounding  box.center)]
\matrix(m)[matrix of math nodes, row sep=2.2em, column sep=4em,
text height=2.5ex, text depth=0.25ex]{
\FU_{\alpha} & \widetilde{\FU}_{\alpha} \\
\FX^{(p)}   &  \\
\widetilde{\FX^{(p)}}         & \widetilde{\FY} \\}; 
\path[->,font=\scriptsize,>=angle 90]
(m-1-1) edge node[left] {$F_{\FX/\FY}|_{\FU_{\alpha}}$} (m-2-1)
(m-1-2) edge node[auto] {$\widetilde{F}_\alpha$} (m-3-1)
        edge (m-3-2)
(m-3-1) edge (m-3-2);
\path[right hook->,font=\scriptsize,>=angle 90]
(m-1-1) edge (m-1-2)
(m-2-1) edge (m-3-1);
\end{tikzpicture}
\end{equation*}
By Lemma \ref{varfiocar} for each $\alpha$ there exists a homomorphism of complexes of $\CO_{\FX^{(p)}}|_{\FU_{\alpha}}$-Modules 
\[
\varphi_{\widetilde{F}_{\alpha}}^{\, 1} \colon
\om^{1}_{\FX^{(p)}/\FY}|_{\FU_{\alpha}}  \lto{} F_{\FX/\FY\,*} \om^{1}_{\FX/\FY}|_{\FU_{\alpha}}
\]
that induces the Cartier  isomorphism in $\CH^{1}$.
By Lema \ref{lemadatospegado} we have that,  for each pair of indexes $\alpha,\, \beta$ such that $\FU_{\alpha \beta}:= \FU_{\alpha} \cap \FU_{\beta} \neq \varnothing$ there exists a homomorphism of $\CO_{\FX^{(p)}}|_{\FU_{\alpha \beta}}$-Modules 
\[\phi_{\alpha \beta} \colon
\om^{1}_{\FX^{(p)} /\FY}|_{\FU_{\alpha \beta}} \lto
 F_{\FX/\FY\,*}\CO_{\FX}|_{\FU_{\alpha \beta}}
\] such that:
\begin{equation} \label{datitos1}
\varphi_{\widetilde{F}_{\alpha}}^{\, 1}|_{\FU_{\alpha \beta}}-\varphi^{\, 1}_{\widetilde{F}_{\beta}}|_{\FU_{\alpha \beta}}= F_{\FX/\FY\,*}\hd \circ \phi_{\alpha \beta}
\end{equation}
and such that, for all $\alpha,\, \beta,\, \delta$ with $\FU_{\alpha \beta \delta}:= \FU_{\alpha} \cap \FU_{\beta} \cap \FU_{\delta} \neq \varnothing$:
\begin{equation} \label{datitos2}
\phi_{\alpha \beta}|_{\FU_{\alpha \beta \delta}} + 
\phi_{ \beta \delta}|_{\FU_{\alpha \beta \delta}} = 
\phi_{\alpha \delta}|_{\FU_{\alpha \beta \delta}}.
\end{equation}

Data (\ref{datitos1}) and (\ref{datitos2}) allow to define a morphism of complexes  
\[
\varphi^{1}_{(\FU_{\alpha},\widetilde{F}_{\alpha})} \colon
\om^{1}_{\FX^{(p)}/\FY} [-1] \lto \chC(\{\FU_{\alpha}\}, F_{\FX/\FY\,*}\om^{\, \bullet}_{\FX/\FY})
\] 
of $\CO_{\FX^{(p)}}$-Modules in degree $1$, equivalently 
\[
\om^{1}_{\FX^{(p)}/\FY} [-1]\xto{(\varphi^1_{1},\varphi^1_{0})} \chC^{1}(\{\FU_{\alpha}\}, F_{\FX/\FY\,*}\CO_{\FX})  \bigoplus \chC^{0}(\{\FU_{\alpha}\}, F_{\FX/\FY\,*}\om^{1}_{\FX/\FY})
\] 
that is locally given by: 
\begin{equation*}
\varphi^1_{1} w _{(\alpha, \beta)} := \phi_{\alpha \beta}(w|_{\FU_{\alpha \beta}}) \qquad 
\varphi^1_{0} w _{(\alpha)} := \varphi^{1}_{\widetilde{F}_{\alpha}}(w|_{\FU_{\alpha}})
\end{equation*}
We define $\varphi^{1}$ as the composition of the morphisms in $\D(\FX^{(p)})$
\[
\om^{1}_{\FX^{(p)}/\FY} [-1]  \xto{\varphi^{1}_{(\FU_{\alpha},\widetilde{F}_{\alpha})}} \chC(\{\FU_{\alpha}\}, F_{\FX/\FY\,*}\om^{\, \bullet}_{\FX/\FY}) \xto{\epsilon^{-1}}  F_{\FX/\FY\,*}  \om^{\bullet}_{\FX/\FY},
\]
where $F_{\FX/\FY\,*}  \om^{\bullet}_{\FX/\FY} \xto{\epsilon}   \chC(\{\FU_{\alpha}\}, F_{\FX/\FY\,*}\om^{\, \bullet}_{\FX/\FY})$ is the  \v{C}ech resolution.
The morphism $\varphi^{1}$ does not  depend of the election of $\{(\FU_{\alpha}, \widetilde{F}_{\alpha})\}$. Indeed, if $\{\FU'_{\beta}\}$ is a refinement of $\{\FU_{\alpha}\}$ it is easy to see that $\varphi^{1}_{(\FU_{\alpha},\widetilde{F}_{\alpha})} = \varphi^{1}_{(\FU'_{\alpha},\widetilde{F}_{\alpha}|_{\FU'_{\alpha}})}$. Then if $\{\FV_{\beta}\}$ is another covering of $\FX$ and for all $\beta$, $\widetilde{G}_{\beta}$ is a  lifting of $F_{\FX/\FY}|_{\FV_{\beta}}$,   is a simple  exercise to check that $\varphi^{1}_{(\FU_{\alpha},\widetilde{F}_{\alpha})} = \varphi^{1}_{(\FU_{\alpha},\widetilde{F}_{\alpha}) \sqcup (\FV_{\beta},\widetilde{G}_{\beta})} =\varphi^{1}_{(\FV_{\beta},\widetilde{G}_{\beta})}$. 

Last, let us see  that $\varphi^{1}$ induces the Cartier isomorphism in $\CH^{1}$. Since it is a  local question, we may suppose that there exists a $\widetilde{\FY}$-morphism $\widetilde{F}\colon  \widetilde{\FX} \to \widetilde{\FX^{(p)}}$ that lifts to $F_{\FX/\FY}$. Then $\varphi^{1}$ is defined  by the morphism $\varphi^{1}_{\widetilde{F}}$ given in \ref{deffisuper1}.
\end{proof}

\begin{cor}
There is a decomposition of $\tau^{\leq 1}(F_{\FX/\FY\,*} \om^{\bullet}_{\FX/\FY})$  in $\D(\FX^{(p)})$.
\end{cor}

\begin{proof}
 Indeed, the maps $\varphi^{0} = F_{\FX/\FY}^\sharp$ and $\varphi^{1}$ provide such isomorphism.
\end{proof}

\begin{prop} \label{propdescomtruncadop} 
There is a decomposition  of $\tau^{<p}(F_{\FX/\FY\,*} \om^{\bullet}_{\FX/\FY})$  in $\D(\FX^{(p)})$ extending the previous one.
\end{prop}
\begin{proof}

For all $1\le i<p$ we're going to find a morphism in $\D(\FX^{(p)})$
\[
\varphi^{i} \colon \omi_{\FX^{(p)}/\FY}[-i] \lto F_{\FX/\FY\,*} \om^{\, \bullet}_{\FX/\FY}
\] 
that induces the Cartier isomorphism through the functor $\CH^{i}$ . 

For that, given $\varphi^{1}$ the morphism defined in Proposition  \ref{prodescomtruncadodos}, for all $i \ge 1$ 
we consider  the  morphism in $\D(\FX^{(p)})$, 
\[
(\varphi^{1})^{\overset{\mathbb{L}}\otimes i} \colon
(\om^{1}_{\FX^{(p)}/\FY})^{\overset{\mathbb{L}}\otimes i} \lto
(F_{\FX/\FY\,*} \om^{1}_{\FX/\FY})^{\overset{\mathbb{L}}\otimes i}
\]
defined, as usual, by $(\varphi^{1})^{\overset{\mathbb{L}}\otimes i}:=\varphi^{1} \overset{\mathbb{L}}\otimes \cdots\overset{\mathbb{L}}\otimes \varphi^{1}$.

By  \cite[Proposition 2.6.1]{LNS} we have that $\om^{1}_{\FX^{(p)}/\FY}$ is a locally free $\CO_{\FX^{(p)}}$-Module of finite rank, then $(\om^{1}_{\FX^{(p)}/\FY}[-1])^{\overset{\mathbb{L}}\otimes i} \cong (\om^{1}_{\FX^{(p)}/\FY})^{\otimes i} [-i]$ in $\D(\FX^{(p)})$. On the other hand, Corollary \ref{corfrobfinitplan} implies that $F_{\FX/\FY\,*} \om^{\, \bullet}_{\FX/\FY}$ is a complex of locally free $\CO_{\FX^{(p)}}$-Modules of finite rank, from which it follows that, in $\D(\FX^{(p)})$,
$(F_{\FX/\FY\,*} \om^{\, \bullet}_{\FX/\FY})^{\overset{\mathbb{L}}\otimes i} \cong (F_{\FX/\FY\,*} \om^{\, \bullet}_{\FX/\FY})^{\otimes i}.$

Take $1\le i < p$. The  antisymmetrization morphism
\[ 
\begin{array}{ccc}
\omi_{\FX^{(p)}/\FY}[-i]& \xto{\quad A \quad} &(\om^{1}_{\FX^{(p)}/\FY})^{\otimes i} [-i]\\
w_{1} \wedge w_{2} \wedge \cdots \wedge w_{i}& \leadsto & \frac{1}{i!} \sum_{\sigma  \in \mathrm{S}_{i}} \mathrm{sg}(\sigma) w_{\sigma(1)} \otimes w_{\sigma(2)} \otimes \cdots \otimes w_{\sigma(i)}\\
\end{array}
\]
is a section of the product map 
\[
\begin{array}{ccc}
(\om^{1}_{\FX^{(p)}/\FY})^{\otimes i} [-i]& \xto{\quad \text{prod.} \quad} 
& \omi_{\FX^{(p)}/\FY}[-i] \\
w_{1} \otimes w_{2} \otimes \cdots \otimes w_{i} & \leadsto 
& w_{1} \wedge w_{2} \wedge \cdots \wedge w_{i}\\
\end{array}
\]
and, then we define $\varphi^{i}$ as the composition of morphisms  in $\D(\FX^{(p)})$:
\begin{equation*}
\begin{tikzpicture}[baseline=(current  bounding  box.center)]
\matrix(m)[matrix of math nodes, row sep=2.5em, column sep=4em,
text height=2ex, text depth=1.25ex]{
\omi_{\FX^{(p)}/\FY}[-i] & (F_{\FX/\FY\,*} \om^{\, \bullet}_{\FX/\FY})\\  
(\om^{1}_{\FX^{(p)}/\FY})^{\otimes i}[-i] & (F_{\FX/\FY\,*} \om^{\, \bullet}_{\FX/\FY})^{\otimes i}\\  
(\om^{1}_{\FX^{(p)}/\FY})^{\overset{\mathbb{L}}\otimes i}[-i]&(F_{\FX/\FY\,*} \om^{\, \bullet}_{\FX/\FY})^{\overset{\mathbb{L}}\otimes i}\\  
}; 
\path[->,font=\scriptsize,>=angle 90]
(m-1-1) edge         node[left] {$A$} (m-2-1)
        edge[dashed] node[auto] {$\varphi^{i}$} (m-1-2)
(m-2-1) edge node[left] {$\wr$} (m-3-1)
(m-3-1) edge node[auto] {$(\varphi^{1})^{\overset{\mathbb{L}}\otimes i}$} (m-3-2)
(m-3-2) edge node[right] {$\wr$} (m-2-2)
(m-2-2) edge node[right] {prod.} (m-1-2);
\end{tikzpicture}
\end{equation*}
From Proposition \ref{prodescomtruncadodos} and Theorem \ref{teorisocar} we conclude that $\CH^{i} (\varphi^{i}) = \gamma^{i}$, where $\gamma^{i}$ is the Cartier isomorphism in degree $i$, for all $0\le i<p$ and with this we end the proof of Theorem \ref{teoremadedescomposicion}.
\end{proof}

\section{Decomposition at $p$}\label{secatp}

\begin{cosa} \textbf{Some reminders on duality on formal schemes.}

Let us recall the definition of some functors involved in the Torsion Duality for formal schemes \cite[\S 6]{AJL}. Given $\FX \in \sfn$ and  $\CJ$ any Ideal of definition of $\FX$, the functor $\varGamma'_{\FX} \colon \A(\FX) \to \A(\FX)$ is  defined by 
\[
\varGamma'_{\FX}(\CF):= \dirlim{n>0} \shom_{\CO_{\FX}}(\CO_{\FX}/\CJ^{n}, \CF).
\]
It is a left exact functor.  The $\CO_{\FX}$-Modules invariant by $\varGamma'_{\FX}$ are called torsion $\CO_{\FX}$-Modules and we denote by $\D_{\qct}(\FX) \subset \D(\FX)$ the full subcategory of complexes such that the homologies are torsion quasi-coherent sheaves. $\widetilde{\D}_{\qc}(\FX):= \R{}{\varGamma'_{\FX}}^{-1}(\D_{\qct}(\FX))$  \cite[Definition 5.2.9]{AJL} 

 The functor $\R{}\varGamma'_{\FX}$ has a right adjoint, the completion functor denoted by $\LLambda_{\FX} \colon \D(\FX) \to \D(\FX)$.  It is given by $\LLambda_{\FX} := \rshom(\R{}\varGamma'_{\FX}\CO_{\FX}, -)$.
See \cite[5.2.10.(3)]{AJL}. The essential image of $\D_{\qct}(\FX)$ through $\LLambda_{\FX}$ is denoted $\widehat{\D}(\FX)$. Note that $\D^{+}_{\cc}(\FX) \subset \widehat{\D}(\FX)$ \cite[Proposition 6.2.1]{AJL}.

Let $f \colon \FX \to \FY$ be a separated map in $\sfn$.
The functor $\R{}f_{*}\colon \D_{\qct}(\FX) \to \D_{\qct}(\FY) \to\D(\FY)$ has a right adjoint, namely $f^{\times}_{\ts} \colon \D(\FY) \to \D_{\qct}(\FX)$ \cite[Theorem 6.1]{AJL}.

Put $f^{\ush}:= \LLambda_{\FX} f^{\times}_{\ts} \colon \D(\FY)  \to \widetilde{\D}_{\qc}(\FX)$. 
The theory of torsion duality associates to  $f $  an adjunction
 \[
 \Hom_{\FX}(\CG, f^{\ush}\CF) \liso \Hom_{\FY}(\R{}f_*\R{}\varGamma'_{\FX}\CG, \CF)
 \]
 with $\CG \in \widetilde{\D}_{\qc}(\FX)$ and $\CF \in \D(\FY)$ induced by  natural transformation (the counit of the adjunction) 
 \[
 \tau^{\ush} \colon \R{}f_*\R{}\varGamma'_{\FX} f^{\ush} \lto \mathsf{id}
 \]
 by \cite[Corollary 6.1.4.(a)]{AJL}.  
\end{cosa}

\begin{cosa} \textbf{Duality for coherent coefficients in the adic case.}

Assume that $f$ is a proper morphism, therefore adic. The above duality is described on the categories $\D^{+}_{\cc}(\FX)$ and $\D^{+}_{\cc}(\FY)$  as follows (see  \cite[Theorem 8.4]{AJL}).  
 The functor $\R{}f_*\R{}\varGamma'_{\FX}$ takes values in $\D_{\qct}(\FY)$ but we may force it to take image on $\widehat{\D}(\FY)$ by applying the completion functor $\LLambda_{\FY}$.  The functor $\LLambda_{\FY}\R{}f_*\R{}\varGamma'_{\FX} $ has the right adjoint $f^{\ush}$. Since $f$ is adic,  using the fact that
\begin{align*}
\LLambda_{\FY}\R{}f_*\R{}\varGamma'_{\FX} &\cong 
\R{}f_*\LLambda_{\FX}\R{}\varGamma'_{\FX} \tag{by \cite[Corollary 5.2.11.(c)]{AJL}} \\
            &\cong  \R{}f_*\LLambda_{\FX} \tag{by \cite[Remarks  6.3.1.(1)(c)]{AJL}}
\end{align*}
we see that $\LLambda_{\FY}\R{}f_*\R{}\varGamma'_{\FX} $ agrees with $\R{}f_{*}$ on $\D^{+}_{\cc}(\FX)$ because the functor $\LLambda_{\FX}|_{\D^{+}_{\cc}(\FX)}$ is the identity.
 
Moreover, by \cite[Proposition 3.5.1, Proposition 8.3.2]{AJL} $\R{}f_{*}(\D^{+}_{\cc}(\FX))\subset \D^{+}_{\cc}(\FY)$ and $f^{\ush}(\D^{+}_{\cc}(\FY)) \subset \D^{+}_{\cc}(\FX)$. 
Therefore the duality for proper morphism establish that the  functor $f^{\ush}\colon \D^{+}_{\cc}(\FY) \to \D^{+}_{\cc}(\FX)$ is right-adjoint to $\R{}f_{*}\colon \D^{+}_{\cc}(\FX) \to\D^{+}_{\cc}(\FY)$. We denote the counit of the adjunction as
 \[
 \widehat{\tau}^{\ush} \colon \R{}f_* f^{\ush} \lto \mathsf{id}.
 \]
 This map is usually referred to as the \emph{trace} map. If we need to specify the map $f$ we will denote it by $\widehat{\tau}^{\ush}_f = \widehat{\tau}^{\ush}$.
\end{cosa}

\begin{cosa} \textbf{Frobenius and a perfect pairing of differential Modules.}

Let $\FX$ denote a smooth pseudo proper formal scheme over a characteristic $p$ perfect field $k$. Let $\dim(\FX) = n$. As before, put $\FX^{(p)} = \FX \times_{F_{k}} \spec(k)$.

 Recall from \ref{cbasepei} the graded complex of coherent $\CO_{\FX}$-Modules $\om^{\bullet}_{\FX/k}$. As we have already recalled (\ref{prodextlocallibr}), the sheaves $\om^{i}_{\FX/k}$ are locally free for all $i$ and thus we have perfect pairings
 \begin{equation*}
 \om^{\, i}_{\FX/k} \otimes_{\CO_{\FX}} \om^{n-i}_{\FX/k} \lto \om^{\, n}_{\FX/k}
 \end{equation*}
 where $0 \leq i \leq n$. This pairing induces the isomorphism in $\D^{+}_{\cc}(\FX)$:
 \begin{equation}\label{pair}
 \om^{\, i}_{\FX/k} \cong  \rshom_{\FX}(\om^{n-i}_{\FX/k} , \om^{\, n}_{\FX/k})
 \end{equation}
 Let us denote $f \colon \FX \to \spec(k)$ and $f^{(p)} \colon \FX^{(p)} \to \spec(k)$ the structural morphisms, and $F_{\FX/k} \colon \FX \to \FX^{(p)}$ the relative Frobenius.  Notice that $F_{\FX/k}$ is a finite map.  Recall that $f^{(p)}\circ F_{\FX/k} = f$.
We have the following string of isomorphisms in $\D^{+}_{\cc}(\FX^{(p)})$:
\begin{align*}
 F_{\FX/k\,*}\om^{\, i}_{\FX/k} & 
         \cong F_{\FX/k\,*} \rshom_{\FX}(\om^{n-i}_{\FX/k}, \om^{\, n}_{\FX/k})\\
                   & =  F_{\FX/k\,*} \rshom_{\FX}(\om^{n-i}_{\FX/k} , 
         \omega_{\FX/k})\\
                   & \cong  F_{\FX/k\,*} \rshom_{\FX}(\om^{n-i}_{\FX/k} , 
         F_{\FX/k}^{\ush}\omega_{\FX^{(p)}/k})\\
                   & \cong \rshom_{\FX^{(p)}} (F_{\FX/k\,*}\om^{n-i}_{\FX/k},
         \omega_{\FX^{(p)}/k})
\end{align*}
Where the first isomorphism comes from applying the functor $F_{\FX/k\,*}$ to (\ref{pair}). The equality corresponds to the notation $\omega_{\FX/k} := \om^{n}_{\FX/k}$; also, we set $\omega_{\FX^{(p)}/k} := \om^{n}_{\FX^{(p)}/k}$. By \cite[Theorem 5.1.2]{sastry} these sheaves are dualizing in $\widehat{\D}(\FX)$ and $\widehat{\D}(\FX^{(p)})$, in other words, they are identified with $f^{\ush}(\widetilde{k})$ and $f^{(p) \ush}(\widetilde{k})$, respectively. The second isomorphism is induced by the map
\[F_{\FX/k}^{\ush}\omega_{\FX^{(p)}/k} \liso \omega_{\FX/k}\]
 (\cite[Corollary 6.1.4.(b)]{AJL}). The third isomorphism is \cite[Theorem 8.4]{AJL} applied to $F_{\FX/k}$ which is finite (Proposition \ref{frobfinitplan}), therefore proper.
 
Taking homology, we obtain the perfect pairing in $\A(\FX^{(p)})$
\begin{equation}\label{pairom}
 F_{\FX/k\,*}\om^{\, i}_{\FX/k} \otimes_{\CO_{\FX^{(p)}}} F_{\FX/k\,*}\om^{n-i}_{\FX/k} \lto \omega_{\FX^{(p)}/k}
\end{equation}
Notice that the pairing is induced by the trace map $\widehat{\tau}^{\ush}(\omega_{\FX^{(p)}/k})$.
 \end{cosa}

\begin{cosa}
 The graded piece of the Cartier isomorphism is an isomorphism of locally free sheaves
\[
\gamma^n \colon \om^{n}_{\FX^{(p)}/k} \lto \CH^{n} (F_{\FX/k\,*}\om^{\bullet}_{\FX/k})
\]
 There is a natural map $\nu\colon F_{\FX/k\,*}\om^{n}_{\FX/k} \to \CH^{n} (F_{\FX/k\,*}\om^{\bullet}_{\FX/k})$ that composed with the inverse of $\gamma^n$ yields a canonical morphism
 
\begin{equation}\label{cartdiff}
 C \colon F_{\FX/k\,*}\omega_{\FX/k} \lto \omega_{\FX^{(p)}/k}
\end{equation}
In other words, $C = (\gamma^n)^{-1} \circ \nu$.
\end{cosa}

\begin{prop}\label{cart=tr}
 The map $C$ in \eqref{cartdiff} agrees with $\widehat{\tau}^{\ush}(\omega_{\FX^{(p)}/k})$ for the Frobenius map $F_{\FX/k}$ .
\end{prop}

\begin{proof}
This comes down to a local computation. Let $x \in \FX$ and $\bar{x} \in \FX^{(p)}$ the corresponding point by the bijection of underlying spaces. Denote by $\widehat{\tau}^{\ush}_{F}$ the map $\widehat{\tau}^{\ush}_{F_{\FX/k}}(\omega_{\FX^{(p)}/k})$. We have the following commutative diagram
\begin{equation*}
\begin{tikzpicture}[baseline=(current  bounding  box.center)]
\matrix(m)[matrix of math nodes, row sep=2.6em, column sep=4.5em,
text height=1.5ex, text depth=0.25ex]{
\R^n f^{(p)}_*\omega_{\FX^{(p)}/k}    & \R^n f_* \omega_{\FX/k} & \widetilde{k} \\
\CH^n_{\bar{x}}(\omega_{\FX^{(p)}/k}) & \CH^n_x(\omega_{\FX/k})         \\}; 
\path[->,font=\scriptsize,>=angle 90]
(m-1-1) edge node[auto] {via $\widehat{\tau}^{\ush}_{F}$} (m-1-2)
(m-1-2) edge node[auto] {$\widehat{\tau}^{\ush}_{f}(\widetilde{k})$} (m-1-3)
(m-2-1) edge node[left] {can} (m-1-1)
        edge node[auto] {$\CH^n_x(\widehat{\tau}^{\ush}_{F})$} (m-2-2)
(m-2-2) edge node[right] {\quad $\res_x$} (m-1-3)
        edge node[auto] {can} (m-1-2);
\end{tikzpicture}
\end{equation*}
with $\CH^n_x$ denoting local cohomology at $x$ and similarly $\CH^n_{\bar{x}}$. The square commutes by functoriality and the triangle defines the map $\res_x$. By pseudo-functoriality the horizontal composition is $\widehat{\tau}^{\ush}_{f^{(p)}}(\widetilde{k})$ .

As a consequence, the lower composition is $\res_{\bar{x}}$. Using the computation in \cite[(7.3.6)]{blue} it follows that $\CH^n_x(\widehat{\tau}^{\ush}_{F}) = \CH^n_x(\gamma^n)$. It holds also in our setting because local cohomology only depends on the completion of the corresponding stalks of the structure sheaves. Notice that in \textit{loc.~cit.}\ $\CH^n_x(\gamma^n)$ is denoted $C_x^{-1}$. The claim follows now by the local description of $\gamma^n$. 
\end{proof}

\begin{rem}
For another take on the relationship between the duality trace and the Cartier map $C$, see
\cite[\S 1]{M}. For an explicit computation of the trace in the case of usual schemes and the absolute Frobenius, see \cite[Theorem 3.2.1]{bs}.
\end{rem}

\begin{thm}[Decomposition at $p$] \label{atp}
Let $\FX$ be a smooth pseudo proper locally noetherian formal scheme over a perfect field $k$ of characteristic $p$ such that $\dim (\FX) \leq p$ and that admits a smooth lifting over $W_2(k)$. Then, the complex $F_{\FX/k\,*} \om^{\bullet}_{\FX/k}$ is decomposable in $\D(\FX^{(p)})$.
\end{thm}

\begin{proof}
 We have to show that there is an isomorphism in $\D(\FX^{(p)})$
 \[
\bigoplus_{i \in \ZZ} \om^{i}_{\FX^{(p)}/k}[-i] \liso
 F_{\FX/k\,*} \om^{\bullet}_{\FX/k}.
 \]
We may assume $\FX$ connected. If $\dim (\FX) < p$ then the statement follows from Corollary \ref{cor1descomposicion}. 

Let us assume from now on that $\dim (\FX) = p$, in other words, $n = p$. By Corollary \ref{cor1descomposicion}, the complex $\tau^{<p}(F_{\FX/k\,*} \om^{\bullet}_{\FX/k})$ is decomposed in $\D(\FX^{(p)})$. We have a distinguished triangle
\begin{equation}\label{tritrun}
 \tau^{<p}(F_{\FX/k\,*} \om^{\bullet}_{\FX/k}) \lto 
F_{\FX/k\,*} \om^{\bullet}_{\FX/k} \lto 
 \CH^p(F_{\FX/k\,*} \om^{\bullet}_{\FX/k})[-p] \overset{+}{\lto}
\end{equation}
As $ \tau^{<p}(F_{\FX/k\,*} \om^{\bullet}_{\FX/\FY})$ is decomposed, we only need to check that the morphism
\[
e \colon \CH^p(F_{\FX/k\,*} \om^{\bullet}_{\FX/k})[-p] \lto (\oplus_{i < p}\CH^i(F_{\FX/k\,*} \om^{\bullet}_{\FX/k})[-i])[1]
\]
is zero. Denote by $e_i$ the components of $e$. They satisfy the following
\[
e_i \in \Hom(\CH^p[-p], \CH^i[-i+1]) = \h^{p-i+1}(\FX^{(p)},\shom(\CH^p,\CH^i))
\]
with $\CH^i := \CH^i(F_{\FX/k\,*} \om^{\bullet}_{\FX/k})$. Applying $\tau^{\geq 1}$ to the triangle \eqref{tritrun} we obtain 
\begin{equation*}
 \tau^{\geq 1}\tau^{<p}(F_{\FX/k\,*} \om^{\bullet}_{\FX/k}) \lto 
 \tau^{\geq 1}(F_{\FX/k\,*} \om^{\bullet}_{\FX/k}) \lto 
 \CH^p(F_{\FX/k\,*} \om^{\bullet}_{\FX/k})[-p] \overset{+}{\lto}
\end{equation*}
By Proposition \ref {cart=tr} the pairing \eqref{pairom} induces an isomorphism 
\[
\rshom_{\FX^{(p)}}(F_{\FX/k\,*}\om^{p-i}_{\FX/k}, \omega_{\FX^{(p)}/k}) \cong F_{\FX/k\,*}\om^{\, i}_{\FX/k}.
\]
Using this, we see that $\tau^{\geq 1}(F_{\FX/k\,*} \om^{\bullet}_{\FX/k})$ is decomposed. Then $e_i = 0$ for all $i \neq 0$. Finally, $e_0 \in \h^{p+1}(\FX^{(p)},\shom(\CH^p,\CH^i)) = 0$ because $\dimtop(\FX^{(p)}) = \dimtop(\FX) \leq \dim(\FX) = p$.
\end{proof}

\end{document}